\documentclass[11pt]{article}

\usepackage[T1]{fontenc}
\usepackage{lmodern}
\usepackage[margin=1.08in]{geometry}
\usepackage{amsmath,amssymb,amsthm,mathtools}
\usepackage{microtype}
\usepackage{xcolor}
\usepackage{tikz}
\usetikzlibrary{arrows.meta}
\usepackage{hyperref}

\hypersetup{
  colorlinks=true,
  linkcolor=blue!55!black,
  citecolor=blue!55!black,
  urlcolor=blue!55!black,
  pdftitle={Simultaneous Busemann–Petty and Shephard Volume Comparisons},
  pdfauthor={Artem Zvavitch}
}

\newtheorem{theorem}{Theorem}[section]
\newtheorem{lemma}[theorem]{Lemma}
\newtheorem*{question}{Question}
\theoremstyle{remark}
\newtheorem{remark}[theorem]{Remark}

\newcommand{\R}{\mathbb{R}}
\newcommand{\Sph}{\mathbb{S}}
\newcommand{\dd}{\,\mathrm{d}}
\newcommand{\abs}[1]{\lvert #1\rvert}
\newcommand{\norm}[1]{\lVert #1\rVert}
\newcommand{\pos}[1]{(#1)_{+}}
\providecommand{\subjclass}[2][]{%
  \par\smallskip\noindent
  {\small\textit{#1 Mathematics Subject Classification.} #2\par}}
\providecommand{\keywords}[1]{%
  \par\smallskip\noindent
  {\small\textit{Key words and phrases.} #1\par}}

\title{Simultaneous Busemann--Petty and Shephard Volume Comparisons}
\author{Artem Zvavitch\thanks{The author is supported in part by U.S. National Science Foundation
Grants DMS-2247771 and DMS-2604412.}}
\date{}

\begin{document}
\maketitle

\begin{abstract}
We study the simultaneous Busemann--Petty and Shephard volume comparison
problem: whether comparison of the volumes of all central hyperplane sections and all
orthogonal hyperplane projections determines the ordering of the volumes
of two convex bodies.  For every $n\geq5$, we construct
origin-symmetric convex bodies of revolution $K,L\subset\R^n$ such that
every central hyperplane section and every orthogonal hyperplane projection
of $K$ has strictly smaller  volume than the corresponding
section or projection of $L$, while $|K|>|L|$.  For $n\leq4$, the affirmative
solution of the Busemann--Petty problem shows that the section inequalities
alone imply $|K|\leq|L|$.  Without origin symmetry, we construct such counterexamples in every
dimension $n\geq2$, with one body a nontrivial translate of a Euclidean
ball and the other a noncentrally symmetric body of constant brightness.
\end{abstract}

\subjclass[2020]{Primary 52A20; Secondary 52A38, 52A40}
\keywords{Busemann--Petty problem, Shephard problem, convex
bodies, central hyperplane sections, orthogonal projections, volume
comparison, bodies of revolution, non-origin-symmetric convex bodies,
constant brightness}

%\tableofcontents

\section{Introduction}

We begin with the origin-symmetric setting.  Let $n\geq2$, let
$K,L\subset\R^n$ be origin-symmetric convex bodies, and let
$\xi\in\Sph^{n-1}$.  We denote by $\xi^\perp$ the central hyperplane
orthogonal to $\xi$ and write $K\mid\xi^\perp$ for the orthogonal
projection of $K$ onto this hyperplane.  If a measurable set $A$ is
full-dimensional in its ambient Euclidean space, its volume is denoted by
$|A|$.  When the dimension is not clear from the context, $|A|_k$ denotes
its $k$-dimensional volume.

The \emph{Busemann--Petty problem}, posed in 1956
\cite{BusemannPetty}, asks whether
\begin{equation}\label{eq:BP-intro}
 \bigl|K\cap\xi^\perp\bigr|_{n-1}
 \leq \bigl|L\cap\xi^\perp\bigr|_{n-1}
 \qquad\text{for all }\xi\in\Sph^{n-1}
\end{equation}
implies $|K|\leq|L|$.  The answer is affirmative for $n\leq4$ and
negative for $n\geq5$.  The solution emerged through a sequence of results
including
\cite{LarmanRogers,Ball,Bourgain,Giannopoulos1990,Papadimitrakis,
Lutwak,Gardner1994,GardnerKoldobskySchlumprecht,Zhang1999}.

The projection analogue is \emph{Shephard's problem}, posed in 1964
\cite{Shephard}: does
\begin{equation}\label{eq:Shephard-intro}
 \bigl|K\mid\xi^\perp\bigr|_{n-1}
 \leq \bigl|L\mid\xi^\perp\bigr|_{n-1}
 \qquad\text{for all }\xi\in\Sph^{n-1}
\end{equation}
imply $|K|\leq|L|$?  The answer is affirmative in dimension two, but
negative in every dimension $n\geq3$ by results of Petty and Schneider
\cite{Petty,Schneider1967}.  The Fourier-analytic relationship
between the two problems is discussed in
\cite{KoldobskyRyaboginZvavitch,KoldobskyBook,KoldobskyYaskinBook};
see also \cite{SchneiderBook,GardnerBook} for general background.

It is natural to ask whether imposing the two comparison hypotheses
simultaneously restores volume comparison.

\begin{question}
Suppose that
\begin{equation}\label{eq:problem}
 \left.
 \begin{aligned}
  \bigl|K\cap\xi^\perp\bigr|_{n-1}
   &\leq \bigl|L\cap\xi^\perp\bigr|_{n-1},\\
  \bigl|K\mid\xi^\perp\bigr|_{n-1}
   &\leq \bigl|L\mid\xi^\perp\bigr|_{n-1}
 \end{aligned}
 \right\}
 \qquad\text{for all }\xi\in\Sph^{n-1}.
\end{equation}
Must \eqref{eq:problem} imply $|K|\leq|L|$?
\end{question}

For $n\leq4$, the section inequalities alone give an affirmative answer.
There are also two important affirmative special cases in every dimension.
Lutwak's intersection-body formulation of the Busemann--Petty problem
shows that \eqref{eq:BP-intro} is sufficient when $K$ is an intersection
body \cite{Lutwak,Gardner1994}, while Schneider's projection-body
formulation of Shephard's problem shows that \eqref{eq:Shephard-intro} is
sufficient when $L$ is a projection body \cite{Schneider1967,Petty}.
 Consequently, an origin-symmetric
counterexample to the simultaneous problem must have $K$ outside the class
of intersection bodies and $L$ outside the class of projection bodies.  In
particular, neither comparison body can be a Euclidean ball, since the ball
belongs to both classes.  Thus the classical counterexamples in which one
of the two bodies is a ball cannot be used directly.

Even without the projection inequalities, the section comparison controls
volume up to an absolute constant.  This is the isomorphic
Busemann--Petty problem, which asks whether
\begin{equation}\label{eq:isomorphic-intro}
 |K|\leq C|L|
\end{equation}
holds with a constant $C$ independent of the dimension.  It is equivalent
to Bourgain's slicing problem; see \cite{MP89}.  Klartag and Lehec resolved
the slicing problem affirmatively \cite{KlartagLehec}; alternative proofs
were subsequently given by Bizeul \cite{Bizeul} and Brazitikos
\cite{BrazitikosSlicing}.  The question considered here is whether adding
all the projection inequalities improves the constant in
\eqref{eq:isomorphic-intro} to exactly one.

A related theorem of Giannopoulos and Koldobsky
\cite{GiannopoulosKoldobsky}, answering a question of V.~Milman, gives the
cross-comparison
\begin{equation}\label{eq:Milman-intro}
 \bigl|K\mid\xi^\perp\bigr|_{n-1}
 \leq
 \bigl|L\cap\xi^\perp\bigr|_{n-1}
 \qquad\text{for all }\xi\in\Sph^{n-1},
\end{equation}
which implies $|K|\leq|L|$, without any symmetry assumption. This makes
the negative result in Theorem \ref{thm:simultaneous} below particularly unexpected.

Hosle studied the reverse mixed comparison, in which the central sections
of $K$ are bounded by the corresponding projections of $L$.  After first
obtaining a $\sqrt n$ estimate for origin-symmetric bodies in John's
position, he recently proved that, for centered $K$, this comparison
implies
$
 |K|\leq c\sqrt n\,|L|,
$
and that the order $\sqrt n$ is sharp
\cite{Hosle,HosleIsomorphic}. 

Our first main theorem gives a negative answer to the simultaneous problem
in every remaining dimension.  Here and below, a \emph{body of revolution}
is a body invariant under every orthogonal transformation that fixes a
distinguished axis pointwise.

\begin{theorem}\label{thm:simultaneous}
For every integer $n\geq5$, there are origin-symmetric convex bodies of
revolution $K,L\subset\R^n$ such that
\[
 \bigl|K\cap\xi^\perp\bigr|_{n-1}
 <\bigl|L\cap\xi^\perp\bigr|_{n-1}
 \quad\text{and}\quad
 \bigl|K\mid\xi^\perp\bigr|_{n-1}
 <\bigl|L\mid\xi^\perp\bigr|_{n-1}
 \qquad\text{for all }\xi\in\Sph^{n-1},
\]
but $|K|>|L|$.
\end{theorem}

Our starting point is the cylinder, whose use in the Busemann--Petty
problem goes back to Giannopoulos \cite{Giannopoulos1990}; see also
\cite{Gardner1994,Papadimitrakis,GrinbergRivin} and the later perturbative
constructions with bodies of revolution in
\cite{GardnerRyaboginYaskinZvavitch,Ryabogin,
NazarovRyaboginZvavitchMaximal,
NazarovRyaboginZvavitchNonuniqueness}.  The cylinder is itself a
projection body, however, and therefore cannot be retained as the
comparison body $L$ with the larger projection data: the projection-body
case of Shephard's theorem would then force $|K|\leq |L|$.  Thus the
cylinder serves only as a limiting model.  We first replace it by a
suitable near-cylinder body $L$, and then construct $K$ through a further
perturbation.  The new point is to find a single two-point formal variation
that controls the section and projection kernels simultaneously.  After
uniformly smoothing this variation and combining it with a small
homothetic contraction, both the section and projection data decrease
while the volume increases.

Origin symmetry is essential for the affirmative conclusion in low
dimensions.  Without symmetry, even uniqueness from the combined section
and projection data was already known to fail.  The antipodal rearrangement
underlying such examples goes back to Gardner and Vol\v{c}i\v{c}
\cite{GardnerVolcic} for sections and to Goodey, Schneider, and Weil
\cite{GoodeySchneiderWeil} for projections.  Ryabogin and Yaskin combined
these ideas to construct, in every dimension $n\geq2$, noncongruent convex
bodies having identical central-section and projection functions
\cite[Lemma~2.1, Proposition~2.2, and
Corollary~2.3]{RyaboginYaskin}.

The construction in \cite{RyaboginYaskin} preserves the unordered
antipodal pairs of radial values and therefore preserves volume.  The
related examples of Nazarov, Ryabogin, and the author in even dimensions
$n\geq4$ also preserve volume
\cite[Theorem~2]{NazarovRyaboginZvavitchNonuniqueness}.  Consequently,
these equal-data, equal-volume examples cannot be directly converted by scaling into
the reversed-volume comparison required here.

In fact, once origin symmetry is omitted, the simultaneous comparison
problem fails in every nontrivial dimension.

\begin{theorem}\label{thm:nonsymmetric}
For every integer $n\geq2$, there are convex bodies of revolution
$K,L\subset\R^n$, both containing the origin in their interiors and neither
origin-symmetric, such that
\[
 \bigl|K\cap\xi^\perp\bigr|_{n-1}
 <\bigl|L\cap\xi^\perp\bigr|_{n-1}
 \quad\text{and}\quad
 \bigl|K\mid\xi^\perp\bigr|_{n-1}
 <\bigl|L\mid\xi^\perp\bigr|_{n-1}
 \qquad\text{for all }\xi\in\Sph^{n-1},
\]
but $|K|>|L|$.  Moreover, $K$ may be chosen to be a nontrivial translate
of a Euclidean ball, while $L$ is not centrally symmetric about any point.
\end{theorem}

The second construction uses a mechanism different from that used in the proof of  Theorem \ref{thm:simultaneous}.  We prescribe an odd
degree-three spherical-harmonic perturbation of the curvature function of
the Euclidean ball and solve the resulting rotational equation explicitly.
This gives a family $L_\varepsilon$ with the same projection function as the
ball, strictly smaller volume, and a uniform second-order estimate for its
central sections.  Comparing $L_\varepsilon$ with a suitably translated and
slightly contracted ball then yields the required strict inequalities and
the reversed volume comparison.

Section~\ref{sec:formulas} derives the formulas and first variations for
bodies of revolution.  Section~\ref{sec:cylinder} proves
Theorem~\ref{thm:simultaneous}, and Section~\ref{sec:nonsymmetric} proves
Theorem~\ref{thm:nonsymmetric}.
\section{Bodies of revolution and first variation formulas}
\label{sec:formulas}

We write $\norm{y}_2$ for the Euclidean norm in $\R^k$ and
$B_2^k=\{y\in\R^k:\norm{y}_2\leq1\}$ for its unit ball.

Let $r:[0,1]\to[0,\infty)$ be a continuous, decreasing and concave function, with $r(0)>0$ and
$r(1)=0$, and define
\begin{equation}\label{eq:Mr}
 M_r=\{(x,z)\in\R\times\R^{n-1}:\abs x\leq1,
                  \ \norm z_2\leq r(\abs x)\}.
\end{equation}
Since $r$ is decreasing and concave, the function
$x\mapsto r(\abs{x})$ is concave on $[-1,1]$ and  $M_r$ is an origin-symmetric
convex body, completely determined by the profile $r$.

Because \(M_r\) is invariant under every orthogonal transformation fixing
the \(x\)-axis pointwise, its central sections and orthogonal projections
depend only on the angle between their normal direction and the \(x\)-axis.
We may therefore parametrize the unit normal by
\begin{equation}\label{eq:tdef}
 \xi_t=\frac{(t,1,0,\ldots,0)}{\sqrt{1+t^2}},
 \qquad 0\leq t<\infty.
\end{equation}
We also use the limiting notation
\(\xi_\infty=(1,0,\ldots,0)\).  Thus \(t=0\) corresponds to a normal
direction perpendicular to the axis of revolution, whereas \(t=\infty\)
corresponds to the axial normal; in the latter direction both the central
section and the orthogonal projection are $r(0)B_2^{n-1}$.

For \(0\leq t\leq\infty\), set
\[
 A_{n,r}(t)
 =\bigl|M_r\cap\xi_t^\perp\bigr|_{n-1},
 \qquad
 P_{n,r}(t)
 =\bigl|M_r\mid\xi_t^\perp\bigr|_{n-1},
\]
where the values at \(t=\infty\) are understood as limits.

\begin{figure}[htbp]
\centering
\begin{tikzpicture}[>=Latex,font=\small]
  \begin{scope}[xshift=3.0cm,x=2.05cm,y=2.05cm]
    \path[fill=blue!7,draw=blue!55!black,very thick]
      plot[domain=-1:1,samples=80] (\x,{0.9*(1-\x*\x)})
      -- plot[domain=1:-1,samples=80] (\x,{-0.9*(1-\x*\x)}) -- cycle;
    \draw[->] (-1.18,0)--(1.25,0) node[right] {$x$};
    \draw[->] (0,-1.08)--(0,1.18) node[above] {$z_1$};
    \draw[densely dashed,thick,black!65]
      (-1.17,{0.55*1.17})--(1.17,{-0.55*1.17})
      node[right] {$\xi_t^\perp$};
    \draw[->,very thick,purple!70!black] (0,0)--(0.45,0.82)
      node[above right=-1pt] {$\xi_t$};
    \node[blue!55!black,fill=white,fill opacity=.82,text opacity=1,
          inner sep=1pt] at (-0.60,0.76) {$z_1=r(|x|)$};
    \node[align=center] at (0,-1.30)
      {\textup{(a) Meridian and direction}};
  \end{scope}

  \begin{scope}[xshift=10.0cm,x=1.65cm,y=1.65cm]
    \path[fill=blue!7,draw=blue!55!black,very thick] (0,0) circle (1);
    \draw[->] (-1.25,0)--(1.28,0) node[right] {$z_1$};
    \draw[->] (0,-1.18)--(0,1.23) node[above] {$z_2$};
    \draw[very thick,purple!70!black] (-0.40,-0.916)--(-0.40,0.916);
    \fill[purple!70!black] (-0.40,-0.916) circle (1.1pt)
                             (-0.40,0.916) circle (1.1pt);
    \draw[<->,black!70] (0,-0.15)--(-0.40,-0.15)
      node[midway,below=3pt,fill=white,inner sep=.5pt,font=\scriptsize]
      {$|tx|$};
    \draw[->,black!75] (0,0)--(0.70,-0.70)
      node[midway,above right=-1pt] {$r(x)$};
    \draw[densely dotted,black!45] (-0.40,0)--(-0.23,0);
    \draw[densely dotted,black!45] (-0.40,0.916)--(-0.23,0.916);
    \draw[<->,black!75] (-0.27,0)--(-0.27,0.916)
      node[midway,right=2pt,fill=white,inner sep=1pt,font=\scriptsize]
      {$\sqrt{r(x)^2-t^2x^2}$};
    \node[purple!70!black] at (-0.40,-1.10)
      {$z_1=-tx$};
    \node[align=center] at (0,-1.55)
      {\textup{(b) Transverse slice for fixed $x\geq0$}};
  \end{scope}
\end{tikzpicture}
\caption{In \textup{(a)}, the
shaded meridian generates $M_r$ by rotation about the $x$-axis; the dashed
line is the trace of $\xi_t^\perp$ in the meridian plane, with unit normal
$\xi_t$.  In \textup{(b)}, for fixed $x\geq0$, the equation $z_1=-tx$ cuts
the transverse $(n-1)$-ball of radius $r(x)$ in an $(n-2)$-ball of the
indicated radius.}
\label{fig:meridian-geometry}
\end{figure}
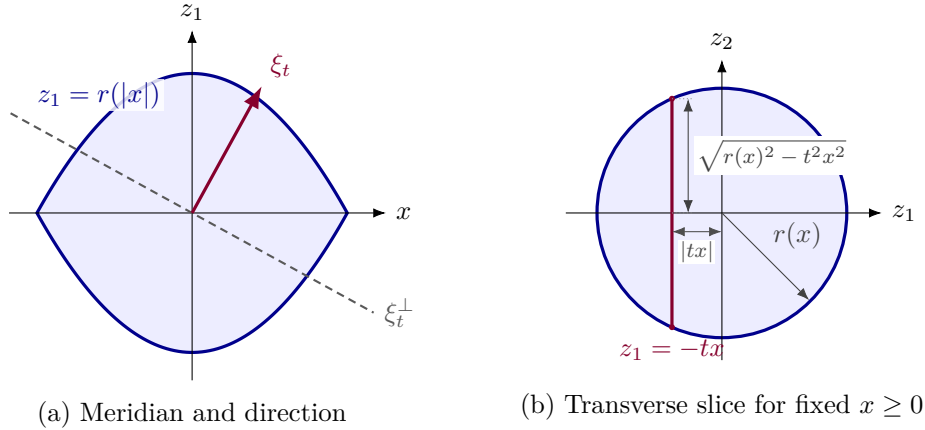
By Fubini's theorem, integration over the slices at axial coordinate $x$
gives the volume formula
\begin{equation}\label{eq:volume-general}
 |M_r|=2|B_2^{n-1}|\int_0^1 r(x)^{n-1}\dd x.
\end{equation}
The following lemma gives the volume of a central hyperplane section. Here and below, $y_+=\max\{y,0\}$.

\begin{lemma}
\label{lem:volume-section-formulas}
Assume $n\geq3$.  For every finite $t\geq0$, the volume of the
corresponding central section of $M_r$ is
\begin{equation}\label{eq:section-general}
 A_{n,r}(t)
 =2|B_2^{n-2}|\sqrt{1+t^2}
 \int_0^1
 \pos{r(x)^2-t^2x^2}^{(n-2)/2}\dd x.
\end{equation}
\end{lemma}

\begin{proof}

To compute the section, write
$z=(z_1,\ldots,z_{n-1})\in\R^{n-1}$.  The equation of $\xi_t^\perp$ is
$tx+z_1=0$, so it is parametrized by
$
 (x,z_2,\ldots,z_{n-1})
 \longmapsto(x,-tx,z_2,\ldots,z_{n-1}).
$
At a fixed $x$, the condition that this point belong to $M_r$ is
\[
 t^2x^2+\sum_{j=2}^{n-1}z_j^2\leq r(\abs{x})^2.
\]
Thus the slice in the variables $z_2,\ldots,z_{n-1}$ is an $(n-2)$-ball of radius
$\sqrt{\pos{r(\abs{x})^2-t^2x^2}}$, see Figure~\ref{fig:meridian-geometry}.  The coordinate vector corresponding
to $x$ is $(1,-t,0,\ldots,0)$, whose length is $\sqrt{1+t^2}$; the other
$n-2$ coordinate vectors are the standard orthonormal vectors in the
$z_2,\ldots,z_{n-1}$ variables.  Hence the induced $(n-1)$-dimensional
volume element on $\xi_t^\perp$ is
$\sqrt{1+t^2}\,\dd x\dd z_2\cdots\dd z_{n-1}$.  Integrating the
volumes of these slices and using symmetry in $x$ proves
\eqref{eq:section-general}.
\end{proof}
Before computing the projection volume, we recall the two equivalent
forms of Cauchy's projection formula; see
\cite[Formula~(A.45), p.~408]{GardnerBook} and
\cite[Chapters~2 and~5]{SchneiderBook}.  If $K$ is a convex body in $\R^n$ and
$\xi\in\Sph^{n-1}$, then
\begin{align}
 \bigl|K\mid\xi^\perp\bigr|_{n-1}
 &=\frac12\int_{\partial K}\abs{\langle n_y,\xi\rangle}
                  \dd\mathcal H^{n-1}(y),
                  \label{eq:Cauchy-projection-boundary}\\
 \bigl|K\mid\xi^\perp\bigr|_{n-1}
 &=\frac12\int_{\Sph^{n-1}}\abs{\langle u,\xi\rangle}\dd S_K(u).
                  \label{eq:Cauchy-projection-sphere}
\end{align}
Here $n_y$ is the outer unit normal to $K$ at almost every
$y\in\partial K$, $\mathcal H^{n-1}$ denotes $(n-1)$-dimensional
Hausdorff measure, and $S_K$ is the surface area measure of $K$: for a
Borel set $\Omega\subset\Sph^{n-1}$, $S_K(\Omega)$ is the
$(n-1)$-dimensional area of the boundary points whose outer unit normal
belongs to $\Omega$.
\begin{lemma}
\label{lem:projection-formula}
Assume $n\geq3$ and, in addition, that
$r\in C^2((0,1))$.  For every finite $t\geq0$,
\begin{equation}\label{eq:projection-prekernel}
 P_{n,r}(t)=\frac{|\Sph^{n-3}|_{n-3}}{\sqrt{1+t^2}}
 \int_0^1r(x)^{n-2}\mathcal J_n(-r'(x)t)\dd x,
\end{equation}
where
\begin{equation}\label{eq:Jndef}
 \mathcal J_n(q)=\int_{-1}^1\abs{q+s}
 (1-s^2)^{(n-4)/2}\dd s
 \qquad\text{for all }q\geq0.
\end{equation}
\end{lemma}

\begin{proof}
We apply \eqref{eq:Cauchy-projection-boundary} to $K=M_r$.
For $x\in(0,1)$, parametrize the two halves of $\partial M_r$ by
$
 X_\pm(x,\omega)=(\pm x,r(x)\omega),$ where $ \omega\in\Sph^{n-2}.
$
Let $\eta_1,\ldots,\eta_{n-2}$ be an orthonormal basis of the tangent
space to $\Sph^{n-2}$ at $\omega$.  The coordinate tangent vectors are
\[
 \partial_xX_\pm=(\pm1,r'(x)\omega),\qquad
 D_{\eta_j}X_\pm=(0,r(x)\eta_j).
\]
Their scalar products are
\[
 \langle\partial_xX_\pm,\partial_xX_\pm\rangle=1+r'(x)^2,
 \qquad
 \langle\partial_xX_\pm,D_{\eta_j}X_\pm\rangle=0,
 \qquad
 \langle D_{\eta_i}X_\pm,D_{\eta_j}X_\pm\rangle
 =r(x)^2\delta_{ij}.
\]
Therefore the Gram determinant is
$(1+r'(x)^2)r(x)^{2(n-2)}$.  The square root of this determinant is the
Jacobian of the parametrization, and hence
\[
 \dd\mathcal H^{n-1}(X_\pm(x,\omega))
 =r(x)^{n-2}\sqrt{1+r'(x)^2}\,\dd x\dd\omega.
\]
 We note that
\[
 n_{X_\pm(x,\omega)}
 =\frac{(\mp r'(x),\omega)}{\sqrt{1+r'(x)^2}}.
\]
If $r'$ is unbounded near $1$, we justify integrability by first
evaluating \eqref{eq:Cauchy-projection-boundary} over
$0<x<1-\delta$, for some small $\delta>0$ and then letting $\delta\downarrow0$.  For each fixed
$t$, the omitted contribution is bounded by a constant multiple of
\[
 \int_{1-\delta}^1r(x)^{n-2}\bigl(1+|r'(x)|\bigr)\dd x,
\]
which tends to zero.  Indeed, the term without $|r'|$ tends to zero by
boundedness of $r$, while
\[
 \int_{1-\delta}^1r(x)^{n-2}|r'(x)|\dd x
 =\frac{r(1-\delta)^{n-1}}{n-1}\longrightarrow0.
\]
The equator $x=0$ and the two tips have zero $(n-1)$-dimensional measure and
do not affect Cauchy's integral.

Combining the formulas for $n_{X_\pm(x,\omega)}$ and the surface
element with \eqref{eq:tdef} gives
\[
 \abs{\langle n_{X_\pm(x,\omega)},\xi_t\rangle}
 \dd\mathcal H^{n-1}
 =\frac{r(x)^{n-2}}{\sqrt{1+t^2}}
   \abs{\mp r'(x)t+\omega_1}\,\dd x\dd\omega.
\]
The spherical integrals from the two halves are equal after the change
$\omega_1\mapsto-\omega_1$.  Thus their sum cancels the factor $1/2$ in
\eqref{eq:Cauchy-projection-boundary}, and
\[
 P_{n,r}(t)=\frac1{\sqrt{1+t^2}}\int_0^1r(x)^{n-2}
 \left(\int_{\Sph^{n-2}}\abs{-r'(x)t+\omega_1}\,\dd\omega\right)\dd x.
\]
Next, for every continuous function $F:[-1,1]\to\R$, spherical
slicing in the first coordinate gives
\[
 \int_{\Sph^{n-2}}F(\omega_1)\dd\omega
 =|\Sph^{n-3}|_{n-3}\int_{-1}^1
 F(s)(1-s^2)^{(n-4)/2}\dd s.
\]
Applying this identity with $F(s)=\abs{-r'(x)t+s}$ proves
\eqref{eq:projection-prekernel}.
\end{proof}
\begin{remark}
Notice that $-r'(x)\geq0$, as required in \eqref{eq:Jndef}.  For the remainder of this section and throughout
Section~\ref{sec:cylinder}, we assume $n\geq5$. In particular, $\mathcal J_n''$ is
well-defined and continuous at $q=1$.  We shall use the identities
\begin{equation}\label{eq:J-identities}
 \mathcal J_n(q)-q\mathcal J_n'(q)
 =\frac{2}{n-2}\pos{1-q^2}^{(n-2)/2},
 \qquad
 \mathcal J_n''(q)=2\pos{1-q^2}^{(n-4)/2}.
\end{equation}

For $0\leq q\leq1$, splitting the integral at $s=-q$ gives
\[
 \mathcal J_n(q)-q\mathcal J_n'(q)
 =2\int_q^1s(1-s^2)^{(n-4)/2}\dd s
 =\frac{2}{n-2}(1-q^2)^{(n-2)/2},
\]
and direct differentiation gives the second identity in
\eqref{eq:J-identities}.  For $q\geq1$, $\mathcal J_n(q)$ is linear, so both
left-hand sides in \eqref{eq:J-identities} vanish.

We also note that, at the two endpoint directions, the section and
projection volumes agree:
\[
\begin{aligned}
 A_{n,r}(0)=P_{n,r}(0)
 &=2|B_2^{n-2}|\int_0^1r(x)^{n-2}\dd x,\\
 A_{n,r}(\infty)=P_{n,r}(\infty)
 &=|B_2^{n-1}|r(0)^{n-1}.
\end{aligned}
\]
\end{remark}
The next lemma gives the first-variation formulas for the volume, sections,
and projections of $M_r$ under a linear perturbation of $r$.
\begin{lemma}\label{lem:first-variations}
Assume $n\geq5$ and $r\in C^2((0,1))$.  Let $g$ be $C^2$ and
compactly supported in $(0,1)$, and define
$r_s(x)=r(x)+s g(x)$.  Suppose that there is $s_0>0$ such that
$r_s$ is nonnegative, decreasing, and concave for every
$0\leq s<s_0$.  Then, for every finite $t\geq0$, the right derivatives at $s=0$ are
\begin{align}
 \left.\frac{\mathrm d}{\mathrm d s}|M_{r_s}|\right|_{s=0}
 &=2(n-1)|B_2^{n-1}|\int_0^1r(x)^{n-2}g(x)\dd x,
 \label{eq:dV-general}\\
 \left.\frac{\mathrm d}{\mathrm d s}A_{n,r_s}(t)\right|_{s=0}
 &=\int_0^1S_{n,x}(t)g(x)\dd x,
 \label{eq:dA-general}\\
 S_{n,x}(t)
 &=2(n-2)|B_2^{n-2}|\sqrt{1+t^2}\,
 r(x)\pos{r(x)^2-x^2t^2}^{(n-4)/2},
 \label{eq:S-kernel}\\
 \left.\frac{\mathrm d}{\mathrm d s}P_{n,r_s}(t)\right|_{s=0}
 &=\int_0^1Q_{n,x}(t)g(x)\dd x,
 \label{eq:dP-general}\\
 Q_{n,x}(t)
 &=\frac{2|\Sph^{n-3}|_{n-3}r(x)^{n-3}}
         {\sqrt{1+t^2}}
 \pos{1-r'(x)^2t^2}^{(n-4)/2}
 \bigl(1-r'(x)^2t^2-r(x)r''(x)t^2\bigr).
 \label{eq:Q-kernel}
\end{align}
\end{lemma}
\begin{remark}
    We note that for $0<x<1$, the projection kernel has the cutoff property $
 Q_{n,x}(t)=0$
 whenever  $|r'(x)|t\geq1$.
When $|r'(x)|t<1$, concavity gives
$
 1-r'(x)^2t^2-r(x)r''(x)t^2\geq0.
$
Consequently, $Q_{n,x}(t)\geq0$ for every $0<x<1$ and $t\geq0$.
\end{remark}
\begin{proof}[Proof of Lemma~\ref{lem:first-variations}] Since $g$ is compactly supported in $(0,1)$, we have $g(0)=g(1)=0$, so
formulas \eqref{eq:volume-general}, \eqref{eq:section-general}, and
\eqref{eq:projection-prekernel} apply to $r_s$.  Differentiating
\eqref{eq:volume-general} under the integral sign gives
\eqref{eq:dV-general}.  
Next, since $y\mapsto\pos{y}^{(n-2)/2}$ is continuously
differentiable for $n\geq5$, including at $y=0$, differentiation of
the section integrand gives
\[
 \left.\frac{\mathrm d}{\mathrm ds}
 \pos{r_s(x)^2-t^2x^2}^{(n-2)/2}\right|_{s=0}
 =(n-2)r(x)g(x)
   \pos{r(x)^2-t^2x^2}^{(n-4)/2}.
\]
The compact support of $g$ supplies an integrable bound, so differentiation
under the integral in \eqref{eq:section-general} is justified.  After
multiplication by its outside factor, this is precisely
\eqref{eq:dA-general}--\eqref{eq:S-kernel}.

For the projection, differentiation under the integral sign is
justified because $g$ and $g'$ are supported in a compact subinterval
of $(0,1)$; on this subinterval all factors in the integrand and their
$s$-derivatives are bounded for $s$ near zero.  Since $
-r_s'(x)=-r'(x)-sg'(x),$ 
the product and chain rules applied to
\eqref{eq:projection-prekernel} give
\begin{align*}
 \left.\frac{\mathrm d}{\mathrm d s}P_{n,r_s}(t)\right|_{s=0}
 &=\frac{(n-2)|\Sph^{n-3}|_{n-3}}{\sqrt{1+t^2}}
 \int_0^1 r(x)^{n-3}\mathcal J_n(-r'(x)t)g(x)\dd x\\
 &\quad-\frac{t|\Sph^{n-3}|_{n-3}}{\sqrt{1+t^2}}
 \int_0^1 r(x)^{n-2}\mathcal J_n'(-r'(x)t)g'(x)\dd x.
\end{align*}
Because $g$ is compactly supported, integration by parts in the second
term produces no boundary term and, after its common positive factor is
removed, gives
\[
 \int_0^1
 \frac{\mathrm d}{\mathrm dx}
 \left( r(x)^{n-2} t\mathcal J_n'(-r'(x)t)\right)g(x)\dd x.
\]
The derivative inside the last integral is
\[
 (n-2) r(x)^{n-3}r'(x)t\mathcal J_n'(-r'(x)t)
 -r(x)^{n-2}r''(x)t^2\mathcal J_n''(-r'(x)t).
\]
Combining this with the first term leaves the integrand
\[
 (n-2)r(x)^{n-3}
   \bigl(\mathcal J_n(-r'(x)t)+r'(x)t\mathcal J_n'(-r'(x)t)\bigr)
 -r(x)^{n-2}r''(x)t^2\mathcal J_n''(-r'(x)t).
\]
Substitution of the two identities in \eqref{eq:J-identities}, followed by
factoring out
$2r(x)^{n-3}\pos{1-r'(x)^2t^2}^{(n-4)/2}$, gives exactly
\eqref{eq:dP-general}--\eqref{eq:Q-kernel}.
\end{proof}

We shall combine a shape perturbation with a small homothety.  Let
$\lambda>0$ be a contraction rate, to be chosen later.  For every
sufficiently small $s\geq0$ for which $r_s=r+sg$ is nonnegative, decreasing,
and concave and $1-\lambda s>0$, define
\begin{equation}\label{eq:homothetic-perturbation}
 K_s=(1-\lambda s)M_{r_s},\qquad L=M_r.
\end{equation}
By homogeneity of the relevant volumes, we obtain
\begin{align}
 \left.\frac{\mathrm d}{\mathrm d s}
 \bigl|K_s\cap\xi_t^\perp\bigr|_{n-1}\right|_{s=0}
 &=\left.\frac{\mathrm d}{\mathrm d s}A_{n,r_s}(t)\right|_{s=0}
   -(n-1)\lambda A_{n,r}(t),
 \label{eq:dAcontract}\\
 \left.\frac{\mathrm d}{\mathrm d s}
 \bigl|K_s\mid\xi_t^\perp\bigr|_{n-1}\right|_{s=0}
 &=\left.\frac{\mathrm d}{\mathrm d s}P_{n,r_s}(t)\right|_{s=0}
   -(n-1)\lambda P_{n,r}(t),
 \label{eq:dPcontract}\\
\left.\frac{\mathrm d}{\mathrm d s}|K_s|\right|_{s=0}
&=\left.\frac{\mathrm d}{\mathrm d s}|M_{r_s}|\right|_{s=0}
   -n\lambda|M_r|.
 \label{eq:dVcontract}
\end{align}
The axial direction $t=\infty$ is covered directly.  Since $g(0)=0$, we
have $r_s(0)=r(0)$, and hence
\begin{equation}\label{common}
 \bigl|K_s\cap e_1^\perp\bigr|_{n-1}
 =\bigl|K_s\mid e_1^\perp\bigr|_{n-1}
 =(1-\lambda s)^{n-1}|B_2^{n-1}|r(0)^{n-1}.
\end{equation}
The derivative of \eqref{common} is
$
 -(n-1)\lambda(1-\lambda s)^{n-2}
 |B_2^{n-1}|r(0)^{n-1},
$
which is negative whenever $1-\lambda s>0$. 
The construction consists of finding $g$ and then choosing $\lambda$ so
that \eqref{eq:dAcontract} and \eqref{eq:dPcontract} are uniformly negative,
while \eqref{eq:dVcontract} is positive.

\section{Origin-symmetric counterexamples in dimensions
\texorpdfstring{$n\geq5$}{n >= 5}}
\label{sec:cylinder}
Fix $n\geq5$.  The construction has two stages.  First, a formal two-point
perturbation of the cylinder determines the required signs and the choice
of parameters.  We then replace the cylinder and the atomic perturbation
by a strictly concave profile and a $C^2$ perturbing function, without
changing these signs.

\subsection{A formal two-point perturbation}
Let $
 C=[-1,1]\times B_2^{n-1}$.
With $\xi_t$ as in \eqref{eq:tdef}, its section and
projection functions are
\begin{align}
 A_C(t)&=\bigl|C\cap\xi_t^\perp\bigr|_{n-1}=2|B_2^{n-2}|\sqrt{1+t^2}\,\Theta_n(t),
 \label{eq:cylinder-section}\\
&\qquad\qquad\mbox{  where } \Theta_n(t)=\int_0^{\min(1,1/t)}
 (1-t^2x^2)^{(n-2)/2}\dd x,
 \label{eq:In}\\
 P_C(t)&=\bigl|C\mid\xi_t^\perp\bigr|_{n-1}=\frac{2|B_2^{n-2}|
                 +|B_2^{n-1}|t}{\sqrt{1+t^2}}.
 \label{eq:cylinder-projection}
\end{align}
In \eqref{eq:In}, we use the convention $1/0=\infty$, so that
$\Theta_n(0)=1$.
Formula~\eqref{eq:cylinder-section} follows from the same slicing calculation
as in the proof of Lemma~\ref{lem:volume-section-formulas}, with the constant
profile $r\equiv1$.  Although this profile does not satisfy the endpoint
condition $r(1)=0$, the section calculation is unchanged.
To show  \eqref{eq:cylinder-projection}, we use \eqref{eq:Cauchy-projection-sphere}. The first term in the
numerator comes from the lateral surface, and the second comes from
the two end caps.  The cylinder
is used here only as a limiting body.  All variations below are supported
away from $x=1$ and therefore leave its end caps unchanged.

The first-variation formulas in Section~\ref{sec:formulas} are kernel
pairings.  Accordingly, if $\mu$ is a finite signed measure supported in
a compact subset of $(0,1)$, we use the formal notation
\begin{align*}
 D_\mu|M_r|
 &={}2(n-1)|B_2^{n-1}|
   \int_{(0,1)}r(x)^{n-2}\dd\mu(x),\\
 D_\mu A_{n,r}(t)
 &={}\int_{(0,1)}S_{n,x}(t)\dd\mu(x),\\
 D_\mu P_{n,r}(t)
 &={}\int_{(0,1)}Q_{n,x}(t)\dd\mu(x).
\end{align*}
For each fixed $t$, these pairings are well defined because the kernels are
continuous on the support of $\mu$.  When $\mu=g(x)\dd x$, they agree with
the usual first variations.

For such a measure $\mu$, define the formal cylinder first variations by
\begin{align*}
 D_\mu A_C(t)
 &={}
 2(n-2)|B_2^{n-2}|\sqrt{1+t^2}
 \int_{(0,1)}
 \pos{1-x^2t^2}^{(n-4)/2}\dd\mu(x),\\
 D_\mu P_C(t)
 &={}
 \frac{2|\Sph^{n-3}|_{n-3}}{\sqrt{1+t^2}}\,
 \mu((0,1)).
\end{align*}
These formulas are obtained from the section kernel and the lateral
projection kernel in \eqref{eq:S-kernel} and \eqref{eq:Q-kernel} by setting
$r=1$ and $r'=r''=0$.  Since $\mu$ is supported in a compact subset of
$(0,1)$, the formal variation has no end-cap contribution.
For an atomic measure, these expressions denote only their linear extensions; no
convex body is being asserted at this stage.  The next lemma identifies a
two-point measure $\mu$ for which the formal section estimate holds.  Its proof
actually identifies a family of such measures; a particular member will be
selected after the additional projection condition has been derived.  The
uniform smooth approximation in
Lemma~\ref{lem:uniform-bump-approximation} below then returns us to an actual
meridian perturbation.

\begin{lemma}\label{lem:cylinder-section}
There exists a signed two-point measure $\mu$, independent of the
dimension, supported in $(0,1)$ and having positive total mass, such that,
for every $n\geq5$ and every $t\geq0$,
\[
 D_\mu A_C(t)\leq(n-2)\mu((0,1))A_C(t).
\]
\end{lemma}

\begin{proof}
Consider measures of the form
\[
 \mu=-\delta_a+(1+m)\delta_b,
 \qquad 0<a<b<1,\quad m>0,
\]
where $\delta_x$ denotes the unit point mass at $x$.  Then
$\mu((0,1))=m$.  By the definition of
$D_\mu A_C$ and \eqref{eq:cylinder-section}, it is enough to prove
\begin{equation}\label{eq:cylinder-section-kernel}
 -\pos{1-a^2t^2}^{(n-4)/2}
 +(1+m)\pos{1-b^2t^2}^{(n-4)/2}
 \leq m\Theta_n(t).
\end{equation}
We first derive a convenient range of $b$ that suffices when $n\geq6$.  If
\begin{equation}\label{eq:cylinder-kernel-bound}
 \pos{1-b^2t^2}^{(n-4)/2}\leq \Theta_n(t)
 \qquad\text{for all }t\geq0,
\end{equation}
then \eqref{eq:cylinder-section-kernel} follows for every $a<b$ and $m>0$,
because
\[
-\pos{1-a^2t^2}^{(n-4)/2}
   +(1+m)\pos{1-b^2t^2}^{(n-4)/2}
\leq m\pos{1-b^2t^2}^{(n-4)/2}
 \leq m\Theta_n(t).
\]
For $0\leq t\leq1$, Jensen's inequality, applied to the convex function
$y\mapsto(1-t^2y)^{(n-2)/2}$, gives
\[
 \begin{aligned}
 \Theta_n(t)
 =\int_0^1(1-t^2x^2)^{(n-2)/2}\dd x\geq
 \left(1-t^2\int_0^1x^2\dd x\right)^{(n-2)/2}
 =\left(1-\frac{t^2}{3}\right)^{(n-2)/2}.
 \end{aligned}
\]
The logarithmic derivative, with respect
to $t^2$, of the ratio
\[
 \frac{(1-t^2/3)^{(n-2)/2}}
      {(1-b^2t^2)^{(n-4)/2}}
\]
has the sign of $
 3(n-4)b^2-(n-2)+2b^2t^2.$ 
Thus the ratio is increasing for every $n\geq6$ whenever $b^2\geq2/3$.
In particular, the desired estimate holds for $0\leq t\leq1$ throughout
the range $b^2\geq3/4$.  This slightly stronger
lower bound on $b$ also gives a simple tail estimate.  Indeed, if
$1\leq t\leq1/b$, then
\[
 \Theta_n(t)=\frac{C_n}{t},\qquad
 C_n=\int_0^1(1-y^2)^{(n-2)/2}\dd y.
\]
For $1\leq t<1/b$,
\[
 \frac{\mathrm d}{\mathrm dt}
 \left[t(1-b^2t^2)^{(n-4)/2}\right]
 =
 (1-b^2t^2)^{(n-6)/2}
 \bigl(1-(n-3)b^2t^2\bigr)<0.
\]
Hence $t(1-b^2t^2)^{(n-4)/2}$ is decreasing on
$[1,1/b]$ by continuity. Moreover,
\[
 C_n\geq\int_0^{1/2}(1-y^2)^{(n-2)/2}\dd y
 \geq\frac12\left(\frac34\right)^{(n-2)/2}
 \geq4^{-(n-4)/2}\geq(1-b^2)^{(n-4)/2}.
\]
This proves \eqref{eq:cylinder-kernel-bound}; for $t\geq1/b$, its left-hand
side vanishes.  Hence every $b$ with $\sqrt3/2\leq b<1$ works for all
$n\geq6$.

We next derive a condition on $a$, $b$, and $m$ ensuring that the same
measure also works when $n=5$.  For $0\leq t\leq1$, the elementary inequalities
\[
 1-y\leq\sqrt{1-y}\leq1-\frac y2
 \mbox{ for all } 0\leq y\leq1
\]
and $(1-y)^{3/2}\geq1-\frac32y$ give
\begin{align*}
 -\sqrt{1-a^2t^2}+(1+m)\sqrt{1-b^2t^2}
 &\leq m+
 \left(a^2-\frac{(1+m)b^2}{2}\right)t^2,\\
 m\Theta_5(t)
 &\geq m\int_0^1
 \left(1-\frac32t^2x^2\right)\dd x
 =m-\frac{mt^2}{2}.
\end{align*}
Consequently, the stronger single condition
\begin{equation}\label{eq:cylinder-n5-condition}
 m+a^2-\frac{(1+m)b^2}{2}<0
\end{equation}
is sufficient.  It proves
\eqref{eq:cylinder-section-kernel} for
$0\leq t\leq1$.  The same upper bound for its left-hand side remains valid
for $1\leq t\leq1/b$; condition
\eqref{eq:cylinder-n5-condition} makes that bound negative at $t=1$ and
decreasing in $t^2$.  For $t\geq1/b$, the positive $b$-term vanishes.

We have therefore shown that every choice
\[
 \frac{\sqrt3}{2}\leq b<1,\qquad 0<a<b,\qquad m>0,
\]
satisfying \eqref{eq:cylinder-n5-condition} has the required property.  This
family is nonempty: for any such $b$, one may first choose
$0<a<b/\sqrt2$ and then take $m>0$ sufficiently small.  This proves the
lemma.  We leave the parameters unfixed until the additional projection
condition is imposed below.
\end{proof}
We now turn to projections.  Let $a$, $b$, $m$, and $\mu$ be any choice
from the family obtained in the proof of
Lemma~\ref{lem:cylinder-section}.  At a flat interior point of the
cylinder, the projection kernel \eqref{eq:Q-kernel} is independent of the
point.  Hence, together with Lemma~\ref{lem:cylinder-section},
\[
 \begin{aligned}
 D_\mu A_C(t)
 &\leq(n-2)mA_C(t),\\
 D_\mu P_C(t)
 &=\frac{2m|\Sph^{n-3}|_{n-3}}{\sqrt{1+t^2}}
 \leq(n-2)mP_C(t),
 \end{aligned}
\]
where $
 |\Sph^{n-3}|_{n-3}=(n-2)|B_2^{n-2}|$. 
We select a particular member of this family after deriving a sufficient
additional condition from the projection kernel of a near-cylinder body.

\subsection{Choosing the parameters and the near-cylinder profile}

We next use the projection kernel to derive a sufficient additional
condition on $a$, $b$, and $m$.  Here $\eta>0$ measures the departure from
the cylinder and will be fixed after the uniform estimates below.  We are looking for
a family of profiles for which, at the two fixed perturbation sites
$x=a,b$,
\[
 r_\eta(x)=1+O(\eta),\qquad
 r_\eta'(x)=O(\eta),\qquad
 r_\eta''(x)=O(\eta)
 \quad\text{as }\eta\downarrow0.
\]
For each fixed $t$, the projection kernel \eqref{eq:Q-kernel} evaluated at
such a profile converges to the flat-cylinder kernel.  The convergence is
not uniform in $t$: when $t$ is of order $\eta^{-1}$, the outside factor
$1/\sqrt{1+t^2}$ is of order $\eta$, whereas the curvature term can be of
order $\eta^{-1}$.  For the profile chosen below, these factors cancel and
leave a nonzero contribution of order one.  We therefore prove a direct
one-sided estimate for the signed two-point combination, uniform in $t$.

The factor
\[
 \pos{1-r_\eta'(x)^2t^2}^{(n-4)/2}
\]
in the projection kernel is nonzero only for
$t<1/|r_\eta'(x)|$.  Since the two-point variation has negative mass at
$a$ and positive mass at $b$, we want
\[
 |r_\eta'(a)|\leq |r_\eta'(b)|.
\]
Then the factor evaluated at $a$ remains nonzero whenever the factor
evaluated at $b$ is nonzero; in particular, the positive contribution
from $b$ can never remain after the negative contribution from $a$ has
vanished.  We also want $-r_\eta''(x)$ to decrease with $x$, so that the negative curvature contribution evaluated at $a$ can
dominate the positive contribution evaluated at $b$.  These two
requirements suggest choosing
\[
 -r_\eta'(x)=\eta\sqrt{x}.
\]
Then
\[
 -r_\eta''(x)=\frac{\eta}{2\sqrt{x}},
\]
and the two quantities have exactly the required monotonicity.

Under the directional rescaling
\[
 t=\frac{\sqrt{\sigma}}{\eta},\qquad \sigma\geq0,
\]
the two relevant factors in \eqref{eq:Q-kernel} become
\[
 \pos{1-r_\eta'(x)^2t^2}^{(n-4)/2}
 =\pos{1-x\sigma}^{(n-4)/2}
\]
and
\[
 -r_\eta(x)r_\eta''(x)t^2
 =\frac{r_\eta(x)}{2\eta\sqrt{x}}\,\sigma.
\]
This also shows explicitly why convergence to the cylinder kernel is not
uniform.  For either $x=a,b$ and fixed $0<\sigma<1/x$, substitution in the
full kernel gives
\begin{align*}
 Q_{n,x}\left(\frac{\sqrt{\sigma}}{\eta}\right)
 &=
 \frac{2|\Sph^{n-3}|_{n-3}r_\eta(x)^{n-3}}
      {\sqrt{1+\sigma/\eta^2}}\,
 (1-x\sigma)^{(n-4)/2}
 \left(
  1-x\sigma+\frac{r_\eta(x)\sigma}{2\eta\sqrt{x}}
 \right)\\
 &\longrightarrow
 |\Sph^{n-3}|_{n-3}
 \sqrt{\frac{\sigma}{x}}\,
 (1-x\sigma)^{(n-4)/2},
\end{align*}
whereas the corresponding flat-cylinder kernel is
\[
 \frac{2|\Sph^{n-3}|_{n-3}}
      {\sqrt{1+\sigma/\eta^2}}
 \longrightarrow0.
\]
Thus the individual near-cylinder kernels do not converge uniformly to
the cylinder kernel.

At $x=a$ and $x=b$, the cutoff factor vanishes for
$\sigma\geq1/a$ and $\sigma\geq1/b$, respectively.  Since $a<b$, we have
$1/a>1/b$.  Thus, whenever the factor evaluated at $b$ is nonzero, the
factor evaluated at $a$ is also nonzero, while for
$1/b\leq\sigma<1/a$ only the negative $a$-term remains.  After the common
factor $\sigma/(2\eta)$ is removed, the leading curvature coefficients at
the two sites are $1/\sqrt a$ and $(1+m)/\sqrt b$.  We are therefore led
to impose the sufficient restriction
\begin{equation}\label{eq:cylinder-boundary-condition}
 \sqrt{\frac ba}>1+m.
\end{equation}
The exact uniform projection estimate resulting from this choice is proved
in Lemma~\ref{lem:projection-boundary-layer}.

We now select the parameters after both sufficient conditions 
\eqref{eq:cylinder-n5-condition} and
\eqref{eq:cylinder-boundary-condition} have been obtained.  The section
estimate allows $b\geq\sqrt3/2$, so we first choose
$b=\sqrt3/2$ and then take $a=b/3$:
\begin{equation}\label{eq:cylinder-parameter-values}
 b=\frac{\sqrt3}{2},\qquad
 a=\frac b3=\frac1{2\sqrt3}.
\end{equation}
With $b^2=3/4$ and $a^2=1/12$, condition
\eqref{eq:cylinder-n5-condition} becomes
\[
 m+a^2-\frac{(1+m)b^2}{2}
 =\frac{5m}{8}-\frac7{24}<0,
\]
or equivalently $m<7/15$.  Moreover,
$1+7/15=22/15<\sqrt3=\sqrt{b/a}$.  Hence both conditions hold for every
\begin{equation}\label{eq:cylinder-m-range}
 0<m<\frac7{15}.
\end{equation}
Fix once and for all any such $m$; this choice may be made independently of
$n$.  We can now fix the signed measure
\begin{equation}\label{eq:cylinder-point-variation}
 \mu_0=-\delta_a+(1+m)\delta_b.
\end{equation}
Then $\mu_0((0,1))=m>0$, and
Lemma~\ref{lem:cylinder-section} applies to this measure.

We next choose the contraction rate.  The formal cylinder section and
projection variations have relative coefficient at most $(n-2)m$, whereas
the normalized volume shape derivative has coefficient $(n-1)m$.

Formally, in accordance with
\eqref{eq:dAcontract}--\eqref{eq:dVcontract}, the contracted first variations
(understood here as the linear extensions evaluated at $\mu_0$) are
\begin{equation}\label{dervsec}
\left. \frac{\mathrm d}{\mathrm d s}
  \bigl|K_s\cap\xi_t^\perp\bigr|_{n-1}\right|_{s=0}
 =D_{\mu_0}A_C(t)-(n-1)\lambda A_C(t) \leq\bigl((n-2)m-(n-1)\lambda\bigr)A_C(t),
 \end{equation}
 \begin{equation}\label{dervproj}
\left.\frac{\mathrm d}{\mathrm d s}
  \bigl|K_s\mid\xi_t^\perp\bigr|_{n-1}\right|_{s=0}
 =D_{\mu_0}P_C(t)-(n-1)\lambda P_C(t)\leq\bigl((n-2)m-(n-1)\lambda\bigr)P_C(t).
\end{equation}
Since $A_C(t)$ and $P_C(t)$ are positive, both contracted derivatives are
strictly negative whenever
\[
 (n-2)m-(n-1)\lambda<0,
 \qquad\text{that is,}\qquad
 \lambda>\frac{(n-2)m}{n-1}.
\]
For volume, $|C|=2|B_2^{n-1}|$ and the formal shape derivative is
\[
 D_{\mu_0}|C|
 =2(n-1)|B_2^{n-1}|m=(n-1)m|C|.
\]
Since volume is homogeneous of degree $n$, the contraction contributes
$-n\lambda|C|$.  Therefore its formal contracted first variation is
\begin{equation}\label{dervvol}
 \left.\frac{\mathrm d}{\mathrm d s}|K_s|\right|_{s=0}
 =D_{\mu_0}|C|-n\lambda|C|=\bigl((n-1)m-n\lambda\bigr)|C|,
\end{equation}
which is strictly positive whenever
\[
 (n-1)m-n\lambda>0,
 \qquad\text{that is,}\qquad
 \lambda<\frac{(n-1)m}{n}.
\]
Combining the two sign conditions, we require
\[
 \frac{(n-2)m}{n-1}
 <
 \lambda
 <
 \frac{(n-1)m}{n}.
\]
To maximize the smaller of the two first-order gaps, we choose $\lambda$
so that they are equal:
\[
 (n-1)\lambda-(n-2)m=(n-1)m-n\lambda.
\]
Thus
\[
 \lambda=\frac{(2n-3)m}{2n-1}.
\]
For this choice, the relative first-variation coefficient for section and
projection volumes at the formal cylinder is
\begin{equation}\label{eq:data-margin}
 (n-2)m-(n-1)\lambda=-\frac{m}{2n-1},
\end{equation}
while the normalized first-variation coefficient for volume is
\begin{equation}\label{eq:volume-margin}
 (n-1)m-n\lambda=\frac{m}{2n-1}.
\end{equation}
Thus the section and projection coefficients are negative, whereas the
volume coefficient is positive, with the same absolute gap.

With $a$, $b$, $m$, $\mu_0$, and $\lambda$ now fixed, we realize the desired
local behavior.  On $[0,1-\eta]$, impose
$r_\eta'(x)=-\eta\sqrt{x}$ and normalize by
$r_\eta(0)=1$.  Integration gives
$r_\eta(x)=1-(2\eta/3)x^{3/2}$ on this interval. 
We close the profile over $[1-\eta,1]$ by subtracting a cubic cap.  Since
$y\mapsto\pos{y}^3$ and its first two derivatives vanish at $y=0$, the cap
joins the uncapped profile in $C^2$ at $x=1-\eta$.  After normalization by
$\eta^3$, the cap equals $1$ at $x=1$.  Its coefficient $1-2\eta/3$ is
precisely the value at $x=1$ of the uncapped profile, and therefore the
resulting profile vanishes at $x=1$.

Accordingly, choose $0<\eta<1-b$.  Then
$0<a<b<1-\eta$, so both perturbation sites lie in the uncapped part of the
profile.  Define
\begin{equation}\label{eq:reta}
 r_\eta(x)=
 \begin{cases}
 \displaystyle
 1-\frac{2\eta}{3}x^{3/2},
 &0\leq x\leq1-\eta,\\[3mm]
 \displaystyle
 1-\frac{2\eta}{3}x^{3/2}
 -\left(1-\frac{2\eta}{3}\right)
  \left(\frac{x-(1-\eta)}{\eta}\right)^3,
 &1-\eta\leq x\leq1.
 \end{cases}
\end{equation}
By construction, $r_\eta$ is $C^2$ on $(0,1)$, and, for $0<x<1$, direct
differentiation gives
\begin{align}
 r_\eta'(x)
 &=-\eta\sqrt{x}
 -\frac{3(1-2\eta/3)}{\eta^3}\pos{x-(1-\eta)}^{\,2}<0,
 \label{eq:reta-prime}\\
 r_\eta''(x)
 &=-\frac{\eta}{2\sqrt{x}}
 -\frac{6(1-2\eta/3)}{\eta^3}\pos{x-(1-\eta)}<0.
 \label{eq:reta-second}
\end{align}
Hence $r_\eta$ is decreasing and strictly concave, with
$r_\eta(0)=1$ and $r_\eta(1)=0$.

Let $L_\eta=M_{r_\eta}$.  Since $r_\eta(x)\geq1-\eta$ for
$0\leq x\leq1-\eta$, while $r_\eta(x)\leq1$ everywhere,
\begin{equation}\label{eq:cylinder-inclusion}
 (1-\eta)C\subset L_\eta\subset C.
\end{equation}
In particular, $|L_\eta|\to|C|$, while the central-section and projection
volumes converge uniformly in direction to the corresponding data of $C$.
More explicitly,
\begin{align*}
 (1-\eta)^{n-1}A_C(t)&\leq A_{n,r_\eta}(t)\leq A_C(t),\\
 (1-\eta)^{n-1}P_C(t)&\leq P_{n,r_\eta}(t)\leq P_C(t)
\end{align*}
for all $t\in[0,\infty]$.

Since $a,b<1-\eta$, both perturbation sites lie in the first branch of
\eqref{eq:reta}.  Thus, for $x=a,b$, we have the exact formulas
\begin{equation}\label{eq:local-exact}
 r_\eta(x)=1-\frac{2\eta}{3}x^{3/2},\qquad
 r_\eta'(x)=-\eta\sqrt{x},\qquad
 r_\eta''(x)=-\frac{\eta}{2\sqrt{x}}.
\end{equation}

Figure~\ref{fig:cylinder-perturbation} illustrates the cylinder
approximation, the cubic cap, and the signed two-point variation.

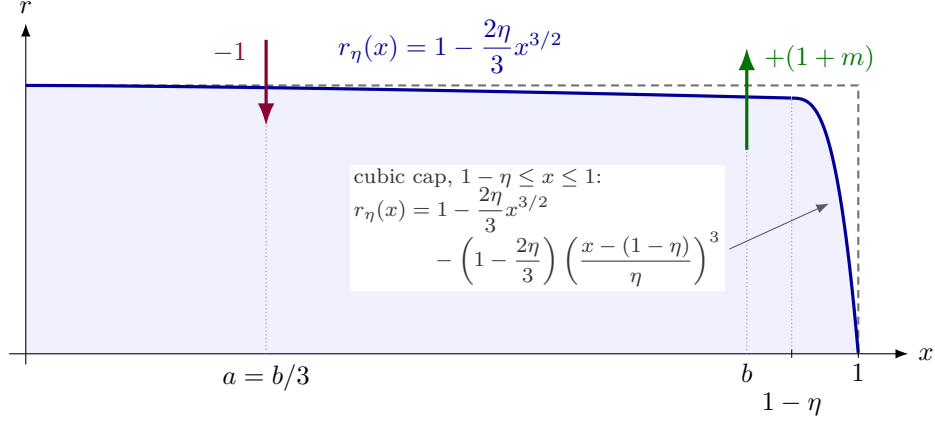
\begin{figure}[htbp]
\centering
\begin{tikzpicture}[x=11.0cm,y=3.55cm,>=Latex,font=\small]
  \path[fill=blue!6]
    (0,0)--(0,1)
    -- plot[domain=0:0.92,samples=70]
       (\x,{1-0.0533333333*(\x)^(1.5)})
    -- plot[domain=0.92:1,samples=45]
       (\x,{1-0.0533333333*(\x)^(1.5)
              -0.9466666667*((\x-0.92)/0.08)^3})
    -- (1,0)--cycle;
  \draw[->] (-0.02,0)--(1.06,0) node[right] {$x$};
  \draw[->] (0,-0.04)--(0,1.23) node[above] {$r$};
  \draw[densely dashed,thick,black!55]
    (0,1)--(1,1)--(1,0);
  \draw[very thick,blue!60!black]
    plot[domain=0:0.92,samples=70]
      (\x,{1-0.0533333333*(\x)^(1.5)})
    plot[domain=0.92:1,samples=45]
      (\x,{1-0.0533333333*(\x)^(1.5)
             -0.9466666667*((\x-0.92)/0.08)^3});
  \node[blue!60!black,fill=white,inner sep=1.5pt] at (0.51,1.145)
    {$r_\eta(x)=1-\dfrac{2\eta}{3}x^{3/2}$};

  \draw[densely dotted,black!45] (0.288675,0)--(0.288675,0.992);
  \draw[densely dotted,black!45] (0.866025,0)--(0.866025,0.957);
  \draw[densely dotted,black!45] (0.92,0)--(0.92,0.953);
  \draw[->,very thick,purple!72!black]
    (0.288675,1.17)--(0.288675,0.85);
  \node[purple!72!black,anchor=east] at (0.278,1.12) {$-1$};
  \draw[->,very thick,green!45!black]
    (0.866025,0.76)--(0.866025,1.14);
  \node[green!45!black,anchor=west] at (0.875,1.10) {$+(1+m)$};
  \node[below] at (0.288675,0) {$a=b/3$};
  \node[below] at (0.866025,0) {$b$};
%  \node[below right=-1pt] at (0.92,0) {$1-\eta$};
%  \node[align=center,black!70] at (0.84,0.43) {cubic\\cap};
%  \draw[->,black!65] (0.875,0.47)--(0.965,0.56);
\draw (0.92,0.015)--(0.92,-0.015);
\draw (1,0.015)--(1,-0.015);
\node[below=11pt] at (0.92,0) {$1-\eta$};
\node[below] at (1,0) {$1$};

\node[align=left,anchor=east,black!80,fill=white,
      inner sep=1.5pt,font=\scriptsize] at (0.84,0.47)
  {cubic cap, $1-\eta\leq x\leq1$:\\[-1pt]
   $\displaystyle r_\eta(x)=1-\frac{2\eta}{3}x^{3/2}$\\[-1pt]
   $\displaystyle\phantom{r_\eta(x)=}
     -\left(1-\frac{2\eta}{3}\right)
      \left(\frac{x-(1-\eta)}{\eta}\right)^3$};

\draw[->,black!65] (0.845,0.39)--(0.965,0.56);
\end{tikzpicture}
\caption{A schematic picture of the cylinder approximation and the signed
two-point variation.
The dashed meridian is the limiting cylinder.  Before the cubic cap, the
solid profile has $r_\eta'(x)=-\eta\sqrt{x}$.  The arrows represent the
weights $-1$ at $a$ and $1+m$ at $b$ in
$\mu_0=-\delta_a+(1+m)\delta_b$.}
\label{fig:cylinder-perturbation}
\end{figure}

\subsection{The projection estimate for the two-point variation}

Here and below, the term \emph{boundary layer} refers to the range of
directions described by the rescaling $t=\sqrt{\sigma}/\eta$; it is not a
physical boundary layer of the body.

For $x=a,b$, remove the common positive factor from
\eqref{eq:Q-kernel} and write
\begin{equation}\label{eq:Psi-eta}
 \Psi_{\eta,x}(t)=r_\eta(x)^{n-3}
 \pos{1-r_\eta'(x)^2t^2}^{(n-4)/2}
 \bigl(1-r_\eta'(x)^2t^2-r_\eta(x)r_\eta''(x)t^2\bigr).
\end{equation}
The following lemma gives the required estimate for the two-point part of
the projection variation.

\begin{lemma}
\label{lem:projection-boundary-layer}
For every $0<\eta<1-b$,
\begin{equation}\label{eq:Psi-projection-bound}
 -\Psi_{\eta,a}(t)+(1+m)\Psi_{\eta,b}(t)\leq m
 \qquad\text{for all }t\geq0.
\end{equation}
\end{lemma}
\begin{proof}
Put $\sigma=\eta^2t^2$.  By \eqref{eq:local-exact}, for $x=a,b$,
\[
 \Psi_{\eta,x}(t)
 =r_\eta(x)^{n-3}\pos{1-x\sigma}^{(n-4)/2}
 \left[
  1+\sigma\left(\frac{r_\eta(x)}{2\eta\sqrt{x}}-x\right)
 \right].
\]
Since $a<b$ and $r_\eta$ is decreasing,
\begin{equation}\label{eq:projection-factor-order}
 r_\eta(a)^{n-3}\pos{1-a\sigma}^{(n-4)/2}
 \geq
 r_\eta(b)^{n-3}\pos{1-b\sigma}^{(n-4)/2}
 \qquad\text{for all }\sigma\geq0.
\end{equation}
For $x=a,b$, formula \eqref{eq:local-exact} gives
\[
 \eta\left(\frac{r_\eta(x)}{2\eta\sqrt{x}}-x\right)
 =\frac1{2\sqrt{x}}-\frac{4\eta}{3}x.
\]
By \eqref{eq:cylinder-parameter-values}, $a=b/3$, and
\eqref{eq:cylinder-m-range} gives $\sqrt{b/a}=\sqrt3>1+m$.  Therefore
\[
\eta\left[
 \left(\frac{r_\eta(a)}{2\eta\sqrt a}-a\right)
 -(1+m)\left(\frac{r_\eta(b)}{2\eta\sqrt b}-b\right)
 \right]
 =
 \frac{\sqrt3-(1+m)}{2\sqrt b}
 +\frac{4b\eta}{3}\left(m+\frac23\right)>0.
\]
Also, since $0<\eta<1-b$,
\[
 \eta\left(\frac{r_\eta(b)}{2\eta\sqrt b}-b\right)
 >\frac1{2\sqrt b}-\frac{4b(1-b)}3>0;
\]
the last inequality follows from $1/(2\sqrt b)>1/2$ and
$4b(1-b)/3\leq1/3$.  Consequently,
\begin{equation}\label{eq:projection-coefficient-order}
 \frac{r_\eta(a)}{2\eta\sqrt a}-a
 >(1+m)\left(\frac{r_\eta(b)}{2\eta\sqrt b}-b\right)>0.
\end{equation}

By \eqref{eq:projection-factor-order}, the terms independent of $\sigma$
satisfy
\[
-r_\eta(a)^{n-3}\pos{1-a\sigma}^{(n-4)/2}
 +(1+m)r_\eta(b)^{n-3}\pos{1-b\sigma}^{(n-4)/2}
 \leq
 m r_\eta(b)^{n-3}\pos{1-b\sigma}^{(n-4)/2}.
\]
By \eqref{eq:projection-factor-order} and
\eqref{eq:projection-coefficient-order}, the terms multiplied by $\sigma$
satisfy
\[
-\left(\frac{r_\eta(a)}{2\eta\sqrt a}-a\right)
 r_\eta(a)^{n-3}\pos{1-a\sigma}^{(n-4)/2}
 +(1+m)
 \left(\frac{r_\eta(b)}{2\eta\sqrt b}-b\right)
 r_\eta(b)^{n-3}\pos{1-b\sigma}^{(n-4)/2}
 \leq0.
\]
Combining the last two estimates, and using $0\leq r_\eta(b)\leq1$, gives
\[
 -\Psi_{\eta,a}(t)+(1+m)\Psi_{\eta,b}(t)
 \leq m r_\eta(b)^{n-3}\pos{1-b\sigma}^{(n-4)/2}
 \leq m.
\]
This proves the lemma.
\end{proof}

\subsection{From the two-point calculation to convex bodies}

We now evaluate the extensions of the first-variation kernels at the
profile $r_\eta$ against the signed measure $\mu_0$.  For sections, after
the common positive geometric factor from the proof of
Lemma~\ref{lem:cylinder-section} is removed, the relevant expression is
\[
 -r_\eta(a)^{n-3}
  \pos{1-\frac{a^2t^2}{r_\eta(a)^2}}^{(n-4)/2}
 +(1+m) r_\eta(b)^{n-3}
  \pos{1-\frac{b^2t^2}{r_\eta(b)^2}}^{(n-4)/2}.
\]
It vanishes for $t\geq1/a$ and converges uniformly on $[0,1/a]$ to the
left-hand side of \eqref{eq:cylinder-section-kernel}.  Consequently, there
is a number $\delta_\eta\geq0$, with $\delta_\eta\to0$ as
$\eta\downarrow0$, such that, for $0\leq t\leq1/a$, the displayed
expression is at most
\[
 m\Theta_n(t)+\delta_\eta.
\]
Since $\Theta_n$ has a positive minimum on $[0,1/a]$, we have
$\delta_\eta=o(1)\Theta_n(t)$ uniformly on this interval.  Restoring the common
geometric factor therefore gives
\[
 D_{\mu_0}A_{n,r_\eta}(t)
 \leq\left((n-2)m+o(1)\right)A_C(t)
 \qquad\text{uniformly in }t.
\]
Finally, \eqref{eq:cylinder-inclusion} gives
$A_C(t)\leq(1-\eta)^{-(n-1)}A_{n,r_\eta}(t)$, and hence
\begin{equation}\label{eq:approx-section-first}
 D_{\mu_0}A_{n,r_\eta}(t)
 \leq\left((n-2)m+o(1)\right)A_{n,r_\eta}(t)
 \qquad\text{uniformly in }t.
\end{equation}
For projections, Lemma~\ref{lem:projection-boundary-layer} gives directly
\[
 D_{\mu_0}P_{n,r_\eta}(t)
 \leq\frac{2m|\Sph^{n-3}|_{n-3}}{\sqrt{1+t^2}}.
\]
Using \eqref{eq:cylinder-inclusion},
\eqref{eq:cylinder-projection}, and
$|\Sph^{n-3}|_{n-3}
=(n-2)|B_2^{n-2}|$, we obtain the uniform relative bound
\begin{equation}\label{eq:approx-projection-first}
 D_{\mu_0}P_{n,r_\eta}(t)
 \leq\frac{(n-2)m}{(1-\eta)^{n-1}}
       P_{n,r_\eta}(t)
 =\left((n-2)m+o(1)\right)P_{n,r_\eta}(t).
\end{equation}
In \eqref{eq:approx-section-first} and
\eqref{eq:approx-projection-first}, $o(1)\to0$ as $\eta\downarrow0$, with
the dimension $n\geq5$ fixed, and the estimates are uniform in $t$.

The volume shape derivative, divided by
$2|B_2^{n-1}|$, is
\[
 (n-1)\bigl(-r_\eta(a)^{n-2}
               +(1+m)r_\eta(b)^{n-2}\bigr)
 \longrightarrow(n-1)m,
\]
while \eqref{eq:cylinder-inclusion} gives
$\int_0^1r_\eta(x)^{n-1}\dd x\to1$.

Now that $n$, $a$, $b$, $m$, and $\lambda$ have been fixed, the uniform
estimates above and
\eqref{eq:data-margin}--\eqref{eq:volume-margin} allow us to choose and fix
$0<\eta<1-b$ so small that
\begin{align}
 D_{\mu_0}A_{n,r_\eta}(t)-(n-1)\lambda A_{n,r_\eta}(t)
 &\leq-\frac{m}{2(2n-1)}A_{n,r_\eta}(t),\label{eq:atomic-A-gap}\\
 D_{\mu_0}P_{n,r_\eta}(t)-(n-1)\lambda P_{n,r_\eta}(t)
 &\leq-\frac{m}{2(2n-1)}P_{n,r_\eta}(t),\label{eq:atomic-P-gap}
\end{align}
uniformly in $t$, whereas the normalized contracted volume functional
\[
 (n-1)\bigl(-r_\eta(a)^{n-2}+(1+m)r_\eta(b)^{n-2}\bigr)
 -n\lambda\int_0^1r_\eta(x)^{n-1}\dd x
 \geq\frac{m}{2(2n-1)}.
\]
Here the first-variation terms are understood as the linear extensions to
the signed measure \eqref{eq:cylinder-point-variation} described above.
At this stage, $\mu_0$ is used only to evaluate these linear functionals;
it does not define an actual meridian profile.

For a body of revolution, the section and projection data depend only on
the angle between the normal direction and the axis of revolution.
Consequently, every unoriented direction is represented by some
$t\in[0,\infty]$ through \eqref{eq:tdef}, where $t=\infty$ represents the
axial normal.  The map
$
 t\longmapsto\frac{t}{\sqrt{1+t^2}}
$
identifies this parameter space with the compact interval $[0,1]$.  The positive continuous functions
$A_{n,r_\eta}$ and $P_{n,r_\eta}$ therefore have positive minima on this
interval.  Hence the relative estimates \eqref{eq:atomic-A-gap} and
\eqref{eq:atomic-P-gap} yield strict additive negative gaps, uniform in the
direction.  We may therefore replace $\mu_0$ by a $C^2$ approximation.
Let
\begin{equation}\label{eq:bump}
 \phi(y)=\frac{35}{32}\pos{1-y^2}^3,\qquad
 \phi_{\varepsilon,c}(x)=\frac1\varepsilon
 \phi\left(\frac{x-c}{\varepsilon}\right),
 \qquad \varepsilon>0,
\end{equation}
so that $\phi$ is nonnegative, $C^2$, supported on $[-1,1]$, and has
integral one.  Put
\begin{equation}\label{eq:gepsilon}
 g_\varepsilon=-\phi_{\varepsilon,a}+(1+m)\phi_{\varepsilon,b}.
\end{equation}
\begin{lemma}
\label{lem:uniform-bump-approximation}
Let $S_{n,x}(t)$ and $Q_{n,x}(t)$ be the kernels
\eqref{eq:S-kernel} and \eqref{eq:Q-kernel}, evaluated at the fixed
profile $r_\eta$.  Then, as $\varepsilon\downarrow0$,
\begin{align*}
 &\sup_{t\geq0}\left|
   \int_0^1S_{n,x}(t)g_\varepsilon(x)\dd x
   -\int_{(0,1)}S_{n,x}(t)\dd\mu_0(x)
  \right|\longrightarrow0,\\
 &\sup_{t\geq0}\left|
   \int_0^1Q_{n,x}(t)g_\varepsilon(x)\dd x
   -\int_{(0,1)}Q_{n,x}(t)\dd\mu_0(x)
  \right|\longrightarrow0,
\end{align*}
and
\[
 \int_0^1r_\eta(x)^{n-2}g_\varepsilon(x)\dd x
 \longrightarrow
 \int_{(0,1)}r_\eta(x)^{n-2}\dd\mu_0(x).
\]
\end{lemma}

\begin{proof}
This is a parameter-uniform version of the standard approximate-identity
argument, \cite[Theorem~8.14]{Folland}.  We include the details
because uniformity with respect to $t$ is essential here.

Choose disjoint closed intervals $E_a,E_b\subset(0,1)$ containing $a$ and
$b$ in their respective interiors, and let $E=E_a\cup E_b$.  For all
sufficiently small $\varepsilon>0$, the supports of
$\phi_{\varepsilon,a}$ and $\phi_{\varepsilon,b}$ are contained in $E_a$
and $E_b$, respectively.  Set
\[
 x_0=\min_{x\in E}x,
 \qquad R_0=\max_{x\in E}r_\eta(x),
 \qquad d_0=\min_{x\in E}\bigl(-r_\eta'(x)\bigr).
\]
All three numbers are positive.  If $t\geq R_0/x_0$, then
$x^2t^2\geq r_\eta(x)^2$ for every $x\in E$, and hence the section kernel
$S_{n,x}(t)$ vanishes on $E$.  Similarly, if $t\geq1/d_0$, then
$r_\eta'(x)^2t^2\geq1$ for every $x\in E$, so the projection kernel
$Q_{n,x}(t)$ vanishes on $E$.  Therefore, if
\[
 T_0=\max\left\{\frac{R_0}{x_0},\frac1{d_0}\right\},
\]
both kernels vanish on $E$ whenever $t\geq T_0$.

It remains to consider the compact set $E\times[0,T_0]$.  Both kernels are
continuous there and hence uniformly continuous.  At a cutoff, this uses
only the continuity of $y\mapsto\pos{y}^{(n-4)/2}$; in dimension five this
is the square-root function $y\mapsto\pos{y}^{1/2}$.  If
$\mathcal K(x,t)$ denotes either kernel, define
\[
 \omega_{\mathcal K}(\varepsilon)
 =\max_{c\in\{a,b\}}
  \sup_{\substack{x\in E,\ |x-c|\leq\varepsilon\\0\leq t\leq T_0}}
  \bigl|\mathcal K(x,t)-\mathcal K(c,t)\bigr|.
\]
Uniform continuity gives
$\omega_{\mathcal K}(\varepsilon)\to0$ as $\varepsilon\downarrow0$.
Since $\phi_{\varepsilon,c}$ is nonnegative, has integral one, and is
supported where $|x-c|\leq\varepsilon$, we have, for $c=a,b$,
\[
 \sup_{t\geq0}\left|
  \int_0^1\mathcal K(x,t)\phi_{\varepsilon,c}(x)\dd x
  -\mathcal K(c,t)
 \right|
 \leq\omega_{\mathcal K}(\varepsilon).
\]
Indeed, the estimate follows from the definition of
$\omega_{\mathcal K}$ for $0\leq t\leq T_0$, while both terms vanish for
$t\geq T_0$.  Taking the linear combination with coefficients $-1$ and
$1+m$ proves the first two assertions of the lemma.

Finally, continuity of $r_\eta^{n-2}$ on $E$ gives, for $c=a,b$,
\[
 \left|
  \int_0^1r_\eta(x)^{n-2}\phi_{\varepsilon,c}(x)\dd x
  -r_\eta(c)^{n-2}
 \right|
 \leq
 \sup_{|x-c|\leq\varepsilon}
 \bigl|r_\eta(x)^{n-2}-r_\eta(c)^{n-2}\bigr|
 \longrightarrow0
\]
as $\varepsilon\downarrow0$.  The same linear combination proves the
volume assertion.
\end{proof}

The contraction terms in \eqref{eq:dAcontract}--\eqref{eq:dVcontract} do
not depend on the approximating measure.  Lemma~\ref{lem:uniform-bump-approximation},
the strict gaps in \eqref{eq:atomic-A-gap}--\eqref{eq:atomic-P-gap}, and
the positive contracted volume margin established above therefore show
that the contracted section and projection derivatives remain uniformly
negative, while the contracted volume derivative remains positive, with
$g_\varepsilon(x)\dd x$ in place of $\mu_0$, provided that
$\varepsilon>0$ is sufficiently small.
 With $\eta$ fixed, fix such an $\varepsilon$.
 
Finally, fix $\eta$ and $\varepsilon$, and let
$D=\operatorname{supp}g_\varepsilon$.  Since $D$ is a compact subset of
$(0,1)$ and
\[
 r_\eta>0,\qquad -r_\eta'>0,\qquad -r_\eta''>0
 \quad\text{on }D,
\]
while $g_\varepsilon$, $g_\varepsilon'$, and
$g_\varepsilon''$ are bounded and vanish outside $D$, compactness shows
that there exists $\tau>0$ such that, for every $0\leq s\leq\tau$, the
profile $r_\eta+s g_\varepsilon$ is nonnegative on $[0,1]$, positive on
$[0,1)$, decreasing, and concave.  After decreasing $\tau$ if necessary,
we may also assume that $1-\lambda s>0$ on this interval.  Write
\[
 r_s=r_\eta+s g_\varepsilon,
 \qquad
 K_s=(1-\lambda s)M_{r_s},
 \qquad 0\leq s\leq\tau.
\]
In the following formulas, the kernels $S_{n,x}(t)$ and $Q_{n,x}(t)$ are
the expressions in \eqref{eq:S-kernel} and \eqref{eq:Q-kernel} evaluated
at the current profile $r_s$.  For every finite $t\geq0$, the exact
derivatives are
\begin{align*}
 \frac{\mathrm d}{\mathrm ds}
 \bigl|K_s\cap\xi_t^\perp\bigr|_{n-1}
 &=(1-\lambda s)^{n-2}
 \left[
  (1-\lambda s)\int_0^1S_{n,x}(t)g_\varepsilon(x)\dd x
  -(n-1)\lambda A_{n,r_s}(t)
 \right],\\
 \frac{\mathrm d}{\mathrm ds}
 \bigl|K_s\mid\xi_t^\perp\bigr|_{n-1}
 &=(1-\lambda s)^{n-2}
 \left[
  (1-\lambda s)\int_0^1Q_{n,x}(t)g_\varepsilon(x)\dd x
  -(n-1)\lambda P_{n,r_s}(t)
 \right],\\
 \frac{\mathrm d}{\mathrm ds}|K_s|
 &=2|B_2^{n-1}|(1-\lambda s)^{n-1}\\
 &\quad{}\times
 \left[
  (n-1)(1-\lambda s)
  \int_0^1r_s(x)^{n-2}g_\varepsilon(x)\dd x
  -n\lambda\int_0^1r_s(x)^{n-1}\dd x
 \right].
\end{align*}

Since $D$ is a compact subset of $(0,1)$, the profiles $r_s$ are uniformly
bounded on $D$, the coordinate $x$ is bounded away from zero there, and
$-r_s'$ is uniformly bounded away from zero.  Consequently, there is a
common finite number $T$, independent of $s\in[0,\tau]$, such that the two
shape integrals in the displayed data derivatives vanish for $t\geq T$.

For $0\leq t\leq T$, the kernel integrands are jointly continuous in
$(s,x,t)$ on the compact set
$[0,\tau]\times D\times[0,T]$.
Continuity at the cutoffs follows from $n\geq5$ and the continuity of
$y\mapsto\pos{y}^{(n-4)/2}$.  The section and projection functions in the
homothetic terms are jointly continuous in $(s,t)$ as well, by
\eqref{eq:section-general} and \eqref{eq:projection-prekernel}.  Hence the
exact contracted data derivatives converge uniformly, for
$0\leq t\leq T$, to their values at $s=0$ as $s\downarrow0$.

At $s=0$, the two data derivatives are bounded above by fixed negative
constants, uniformly in $t$, while the volume derivative is positive.
After decreasing $\tau$, these signs therefore persist for every
$0\leq s\leq\tau$ and $0\leq t\leq T$.  For finite $t\geq T$, the shape
integrals vanish, and the data derivatives are respectively
\[
 -(n-1)\lambda(1-\lambda s)^{n-2}A_{n,r_s}(t)
 \quad\text{and}\quad
 -(n-1)\lambda(1-\lambda s)^{n-2}P_{n,r_s}(t),
\]
so they are strictly negative.  The displayed volume derivative is
continuous in $s$ and remains positive after the same decrease of $\tau$.

The remaining axial direction $t=\infty$ is immediate.  Since
$g_\varepsilon(0)=0$, we have $r_s(0)=r_\eta(0)$, and hence
\[
 \bigl|K_s\cap e_1^\perp\bigr|_{n-1}
 =\bigl|K_s\mid e_1^\perp\bigr|_{n-1}
 =(1-\lambda s)^{n-1}|B_2^{n-1}|r_\eta(0)^{n-1}.
\]
Both axial data are therefore strictly decreasing in $s$.

Integrating these derivatives gives, with $
 L=L_\eta$ and  $K=K_\tau$,
\[
 \begin{gathered}
  \left.
  \begin{aligned}
   \bigl|K\cap\xi^\perp\bigr|_{n-1}
    &<\bigl|L\cap\xi^\perp\bigr|_{n-1},\\
   \bigl|K\mid\xi^\perp\bigr|_{n-1}
    &<\bigl|L\mid\xi^\perp\bigr|_{n-1}
  \end{aligned}
  \right\}
  \qquad\text{for all }\xi\in\Sph^{n-1},\\
  |K|>|L|.
 \end{gathered}
\]
This establishes the counterexample in
Theorem~\ref{thm:simultaneous} for every $n\geq5$.

\begin{remark}
A five-dimensional counterexample cannot be lifted to higher dimensions by
taking products with a common interval.  Indeed, projection of
$K\times[-h,h]$ onto the hyperplane orthogonal to the new coordinate axis
is naturally identified with $K$; hence $|K|>|L|$ would itself violate the
required projection inequality in the lifted dimension.  Our construction
instead rebuilds the meridian in each dimension and controls mixed
directions through the $n$-dimensional kernels
\eqref{eq:S-kernel} and \eqref{eq:Q-kernel}.
\end{remark}

\section{Counterexamples without origin symmetry}\label{sec:nonsymmetric}

Non-spherical convex bodies of constant brightness are classical; see
\cite{FireyConstantBrightness}.  Here we construct explicitly a
one-parameter family $L_\varepsilon$ of nonsymmetric bodies of revolution
near $B_2^n$, tailored to the present comparison problem.  Every projection
of $L_\varepsilon$ has the same $(n-1)$-dimensional volume as the
corresponding projection of $B_2^n$; moreover, for every sufficiently small
$\varepsilon\ne0$, the body $L_\varepsilon$ has smaller volume, and its
central sections satisfy the uniform second-order estimate needed below.
For bodies of revolution, the curvature equation reduces to an ordinary
differential equation that can be solved explicitly; compare the related
one-dimensional reconstruction in
\cite{RyaboginZvavitchReconstruction}.

Let $L$ be a convex body of class $C^2_+$, that is, with $C^2$ boundary
and positive Gaussian curvature, and let
$
 h_L(u)=\max_{x\in L}\langle x,u\rangle,$ $ u\in\Sph^{n-1},
$
be its support function.  Let $\dd u$ denote spherical surface measure on
$\Sph^{n-1}$.  If $I_{u^\perp}$ denotes the identity on tangent space 
$T_u\Sph^{n-1}=u^\perp$, then the density of the surface area measure 
$S_L$, with respect to spherical surface measure $\dd u$ is
\begin{equation}\label{eq:nonsym-curvature-function}
 f_L(u)=
 \det\bigl(\nabla_{\Sph^{n-1}}^2h_L(u)
              +h_L(u)I_{u^\perp}\bigr).
\end{equation}
This density is the curvature function of $L$, and the eigenvalues of the
matrix in \eqref{eq:nonsym-curvature-function} are its principal radii of
curvature; see
\cite[Corollaries~2.5.2 and~2.5.3 and pp.~115--120]{SchneiderBook}.
Together with Cauchy's projection formula  \eqref{eq:Cauchy-projection-boundary} or 
\cite[Formula~(A.45), p.~408]{GardnerBook}, this gives
\begin{equation}\label{eq:nonsym-Cauchy}
 \bigl|L\mid\xi^\perp\bigr|_{n-1}
 =\frac12\int_{\Sph^{n-1}}
       |\langle u,\xi\rangle|f_L(u)\,\dd u.
\end{equation}
Consequently, an odd degree perturbation 
does not change any projection volume.
Moreover, every curvature function satisfies the equilibrium identity
$
\int_{\Sph^{n-1}}u f_L(u),\dd u=0 
$
and therefore has no spherical-harmonic component of degree one. On the
support-function side, adding a harmonic of degree one merely translates
the body. Thus degree three is the first nontrivial odd choice. We will perturb  $f_{B_2^n}(u)=1$ with a spherical harmonic of degree $3$. More particular we use perturbation  
\begin{equation}\label{eq:nonsym-curvature}
 1+\varepsilon\left[6 u_1-2(n+2) u_1^3\right].
\end{equation}
Note that 
\[
 6u_1-2(n+2)u_1^3
 =-2(n+2)\left(u_1^3-\frac{3u_1}{n+2}\right),
\]
and the expression in parentheses is the restriction to the sphere of the
homogeneous harmonic polynomial
$
 x_1^3-\frac{3}{n+2}x_1\norm{x}_2^2;
$
see \cite[Chapter~3]{Groemer}.  The coefficients in
\eqref{eq:nonsym-curvature} are also chosen so that the rotational
curvature equation below has a particularly simple solution.

\begin{lemma}\label{lem:nonsym-perturbation}
For all sufficiently small $|\varepsilon|$, there is a $C^\infty$ strictly
convex body of revolution $L_\varepsilon$, containing the origin in its
interior and having everywhere positive Gaussian curvature, whose curvature
function is
\[
 f_{L_\varepsilon}(u)
 =1+6\varepsilon u_1-2(n+2)\varepsilon u_1^3.
\]
Its support function depends smoothly on $\varepsilon$, with
$L_0=B_2^n$, and
\begin{equation}\label{eq:nonsym-parity}
 L_{-\varepsilon}=-L_\varepsilon.
\end{equation}
Moreover,
\begin{equation}\label{eq:nonsym-projection}
 \bigl|L_\varepsilon\mid\xi^\perp\bigr|_{n-1}
 =|B_2^{n-1}|
 \qquad\text{for every }\xi\in\Sph^{n-1},
\end{equation}
whereas, for $\varepsilon\ne0$,
\begin{equation}\label{eq:nonsym-volume-gap}
 |L_\varepsilon|<|B_2^n|.
\end{equation}
\end{lemma}

\begin{proof}
We first reduce the curvature equation to one variable.  For
$u\in\Sph^{n-1}$, put $\alpha(u)=u_1$ and seek a rotationally invariant
support function of the form
$h_{L_\varepsilon}(u)=H_\varepsilon(\alpha)$.  The spherical gradient of
$\alpha$ is the tangential projection of the Euclidean gradient $e_1$, and
hence
$
 \nabla_{\Sph^{n-1}}\alpha=e_1-\alpha u.
$
Differentiating this identity in tangent directions gives
$
 \nabla_{\Sph^{n-1}}^2\alpha=-\alpha I_{u^\perp}.
$ 
The chain rule therefore yields
\begin{equation}\label{eq:nonsym-zonal-Hessian}
 \nabla_{\Sph^{n-1}}^2h_{L_\varepsilon}
 +h_{L_\varepsilon}I_{u^\perp}
 =H_\varepsilon''(\alpha)
   \nabla_{\Sph^{n-1}}\alpha\otimes
   \nabla_{\Sph^{n-1}}\alpha
  +\bigl(H_\varepsilon(\alpha)-\alpha H_\varepsilon'(\alpha)\bigr)
   I_{u^\perp}.
\end{equation}
For $|\alpha|<1$, we have
$|\nabla_{\Sph^{n-1}}\alpha|^2=1-\alpha^2$.  Thus the eigenvalues in
\eqref{eq:nonsym-zonal-Hessian} are
\begin{equation}\label{eq:nonsym-principal-radii}
 \begin{aligned}
 H_\varepsilon(\alpha)-\alpha H_\varepsilon'(\alpha)\quad
 &\text{with multiplicity }n-2,\\
 H_\varepsilon(\alpha)-\alpha H_\varepsilon'(\alpha)
 +(1-\alpha^2)H_\varepsilon''(\alpha)\quad
 &\text{with multiplicity }1.
 \end{aligned}
\end{equation}
At $\alpha=\pm1$, the spherical gradient of $\alpha$ vanishes, so all $n-1$
eigenvalues equal
$H_\varepsilon(\alpha)-\alpha H_\varepsilon'(\alpha)$; this also follows by
continuity from \eqref{eq:nonsym-principal-radii}.

We now explain how the support function is found.  Our goal is to solve
\begin{equation}\label{eq:nonsym-H-equation}
 \bigl(H_\varepsilon(\alpha)-\alpha H_\varepsilon'(\alpha)\bigr)^{n-2}
 \bigl(H_\varepsilon(\alpha)-\alpha H_\varepsilon'(\alpha)
 +(1-\alpha^2)H_\varepsilon''(\alpha)\bigr)
 =1+6\varepsilon\alpha-2(n+2)\varepsilon\alpha^3.
\end{equation}
Set
$
 q_\varepsilon(\alpha)
 =H_\varepsilon(\alpha)-\alpha H_\varepsilon'(\alpha).
$
Then $q_\varepsilon'(\alpha)=-\alpha H_\varepsilon''(\alpha)$.  If
$w=q_\varepsilon^{n-1}$, the left-hand side of
\eqref{eq:nonsym-H-equation} becomes, for $\alpha\ne0$,
\begin{equation}\label{eq:nonsym-q-equation}
 w(\alpha)-\frac{1-\alpha^2}{(n-1)\alpha}w'(\alpha).
\end{equation}
The expression at $\alpha=0$ is understood by continuity.  After
multiplication by $(n-1)\alpha$, the equation is regular at
$\alpha=0$ and is a first-order linear equation for $w$.  A solution that is smooth on the whole interval $[-1,1]$ is
\begin{equation}\label{eq:nonsym-q}
 q_\varepsilon(\alpha)^{n-1}
 =1-2(n-1)\varepsilon\alpha^3.
\end{equation}
Indeed, substitution into \eqref{eq:nonsym-q-equation} gives
\[
 1-2(n-1)\varepsilon\alpha^3
 +6\varepsilon\alpha(1-\alpha^2)
 =1+6\varepsilon\alpha-2(n+2)\varepsilon\alpha^3.
\]
For completeness,  note that the difference of two solutions of the linear equation
is
$
 C(1-\alpha^2)^{-(n-1)/2},
$
which cannot be smooth at the endpoints unless $C=0$.  Thus smoothness on
$[-1,1]$ supplies the boundary condition that makes
\eqref{eq:nonsym-q} unique.

For sufficiently small $|\varepsilon|$, the right-hand side of
\eqref{eq:nonsym-q} is positive, so we take
\[
 q_\varepsilon(\alpha)
 =\bigl(1-2(n-1)\varepsilon\alpha^3\bigr)^{1/(n-1)}.
\]
Differentiating gives
\[
 q_\varepsilon'(\alpha)
 =-6\varepsilon\alpha^2q_\varepsilon(\alpha)^{-(n-2)}.
\]
Since $q_\varepsilon'=-\alpha H_\varepsilon''$, the required second
derivative is
$
 H_\varepsilon''(\alpha)
 =6\varepsilon\alpha q_\varepsilon(\alpha)^{-(n-2)},
$
where the formula is smooth also at $\alpha=0$.  The conditions
$H_\varepsilon(0)=1$ and $H_\varepsilon'(0)=0$ give
\begin{equation}\label{eq:nonsym-H}
 H_\varepsilon(\alpha)
 =1+6\varepsilon\int_0^\alpha
       (\alpha-s)s\,q_\varepsilon(s)^{-(n-2)}\,\dd s.
\end{equation}
Here $H_\varepsilon(0)=q_\varepsilon(0)=1$ is required by the identity
$q_\varepsilon=H_\varepsilon-\alpha H_\varepsilon'$, while
$H_\varepsilon'(0)=0$ fixes the remaining freedom to add a term
$c\alpha$.  Such a term is the support function of a translation by
$ce_1$ and does not affect the curvature function.  Finally,
\[
 \frac{\dd}{\dd\alpha}
 \bigl(H_\varepsilon-\alpha H_\varepsilon'\bigr)
 =-\alpha H_\varepsilon''=q_\varepsilon',
\]
and both $H_\varepsilon-\alpha H_\varepsilon'$ and $q_\varepsilon$ equal
$1$ at $\alpha=0$.  Hence
\begin{equation}\label{eq:nonsym-q-identity}
 H_\varepsilon(\alpha)-\alpha H_\varepsilon'(\alpha)
 =q_\varepsilon(\alpha),
\end{equation}
as required.

Equations \eqref{eq:nonsym-principal-radii}--
\eqref{eq:nonsym-q-identity} give
\begin{align*}
 \det\bigl(\nabla_{\Sph^{n-1}}^2h_{L_\varepsilon}
              +h_{L_\varepsilon} I_{u^\perp}\bigr)
 &=q_\varepsilon(\alpha)^{n-2}
   \bigl(q_\varepsilon(\alpha)
         +(1-\alpha^2)H_\varepsilon''(\alpha)\bigr)\\
 &=1-2(n-1)\varepsilon\alpha^3
   +6\varepsilon\alpha(1-\alpha^2)\\
 &=1+6\varepsilon\alpha-2(n+2)\varepsilon\alpha^3.
\end{align*}
For small $|\varepsilon|$, both $q_\varepsilon$ and the curvature density
in \eqref{eq:nonsym-curvature} are positive.  If $n\geq3$, the first
principal radius is $q_\varepsilon>0$, and the remaining one is
\[
 \frac{1+6\varepsilon\alpha-2(n+2)\varepsilon\alpha^3}
      {q_\varepsilon(\alpha)^{n-2}}>0.
\]
When $n=2$, this is the only principal radius.  Thus
$\nabla_{\Sph^{n-1}}^2h_{L_\varepsilon}
 +h_{L_\varepsilon}I_{u^\perp}$ is positive definite everywhere.

The explicit formulas show that $h_{L_\varepsilon}$ depends smoothly on
$\varepsilon$, converges to $1$ in every $C^k$ norm, and is therefore
positive for small $|\varepsilon|$.  Its $1$-homogeneous extension
$\widetilde h_{L_\varepsilon}(x)
=|x|h_{L_\varepsilon}(x/|x|)$, $x\ne0$, extended by
$\widetilde h_{L_\varepsilon}(0)=0$, has positive semidefinite Euclidean
Hessian and is convex.  Hence $h_{L_\varepsilon}$ is a support function; see
\cite[Sections~1.7 and~2.5]{SchneiderBook}.  The same criterion shows that
the resulting body $L_\varepsilon$ is $C^\infty$, strictly convex, and has
positive Gaussian curvature.  Positivity of the support function means
that the origin lies in the interior.  Since $h_{L_\varepsilon}(u)$ depends
only on $u_1$, the body is rotationally invariant about the $e_1$-axis.

The formulas also give
$
 q_{-\varepsilon}(\alpha)=q_\varepsilon(-\alpha)$, and 
 $H_{-\varepsilon}(\alpha)=H_\varepsilon(-\alpha)$. 
Consequently $h_{L_{-\varepsilon}}(u)=h_{L_\varepsilon}(-u)$, which is exactly
\eqref{eq:nonsym-parity}.

The nonconstant part of the curvature function is odd.  Therefore
\eqref{eq:nonsym-Cauchy} gives \eqref{eq:nonsym-projection}.  The same oddness shows that $L_\varepsilon$ and $B_2^n$  have the same surface area. For
$\varepsilon\ne0$, the curvature function of $L_\varepsilon$ is
nonconstant, so $L_\varepsilon$ is not a Euclidean ball. The
isoperimetric inequality and its equality condition
\cite[p.~382]{SchneiderBook} now imply
\eqref{eq:nonsym-volume-gap}.
\end{proof}

The following estimate is the only further information about the central
sections of $L_\varepsilon$ that is needed.  Its smooth-factorization argument
is analogous to the one used in
\cite[Lemma~3.3]{NazarovRyaboginZvavitchLocal}. 
Recall that the radial function of a convex body $K$ containing the
origin in its interior is
\[
 \rho_K(u)=\max\{\lambda\geq0:\lambda u\in K\},
 \qquad u\in\Sph^{n-1}.
\]

\begin{lemma}\label{lem:nonsym-section-factorization}
There are constants $\varepsilon_0>0$ and $M>0$ such that
\begin{equation}\label{eq:nonsym-section-factorization}
 \left|
 \frac{|L_\varepsilon\cap\xi^\perp|_{n-1}}{|B_2^{n-1}|}-1
 \right|
 \leq M\varepsilon^2\xi_1^2
\end{equation}
whenever $|\varepsilon|\leq\varepsilon_0$ and $\xi\in\Sph^{n-1}$.
\end{lemma}

\begin{proof}
For $u\in\Sph^{n-1}$, the point of $\partial L_\varepsilon$ with outer
unit normal $u$ is
\begin{equation}\label{eq:nonsym-boundary-map}
 X_\varepsilon(u)=h_{L_\varepsilon}(u)u
                 +\nabla_{\Sph^{n-1}}h_{L_\varepsilon}(u).
\end{equation}
This is the inverse spherical-image parametrization; see
\cite[Section~0.9, especially p.~24]{GardnerBook} and
\cite[Corollary~1.7.3]{SchneiderBook}.  Differentiating
\eqref{eq:nonsym-boundary-map} on $u^\perp$ gives
\[
 D_uX_\varepsilon
 =\nabla_{\Sph^{n-1}}^2h_{L_\varepsilon}(u)
  +h_{L_\varepsilon}(u)I_{u^\perp},
\]
which is positive definite by Lemma~\ref{lem:nonsym-perturbation}.
Thus $X_\varepsilon$ is a smooth nondegenerate parametrization of the
boundary, jointly smooth in $(\varepsilon,u)$.

It remains to show that the radial function depends smoothly on
$\varepsilon$.  Set $
 p_\varepsilon(u)
 =
 X_\varepsilon(u)/\|X_\varepsilon(u)\|_2.
$ 
Because $L_\varepsilon$ is smooth and strictly convex and contains the
origin in its interior, every ray from the origin meets
$\partial L_\varepsilon$ in exactly one point, and that boundary point
has a unique outer unit normal.  Hence $p_\varepsilon$ is bijective.  Moreover,
$\langle X_\varepsilon(u),u\rangle=h_{L_\varepsilon}(u)>0$.  Hence
$X_\varepsilon(u)$ is not tangent to the boundary.  The kernel of the
differential of $x\mapsto x/|x|$ at $X_\varepsilon(u)$ is
$\R X_\varepsilon(u)$, whereas $D_uX_\varepsilon$ maps onto $u^\perp$.
These subspaces meet only at the origin, so the differential of
$p_\varepsilon$ is nonsingular.  The parameterized inverse function theorem
applied to
$(\varepsilon,u)\mapsto(\varepsilon,p_\varepsilon(u))$ therefore shows that
$p_\varepsilon^{-1}(v)$ is smooth in $(\varepsilon,v)$.  Consequently,
\begin{equation}\label{eq:nonsym-radial-from-support}
 \rho_{L_\varepsilon}(v)
 =\bigl|X_\varepsilon\bigl(p_\varepsilon^{-1}(v)\bigr)\bigr|
\end{equation}
is smooth in $(\varepsilon,v)$.

Polar coordinates inside $\xi^\perp$ give
\begin{equation}\label{eq:nonsym-section-Radon}
 \mathcal A(\varepsilon,\xi)
 :=\frac{|L_\varepsilon\cap\xi^\perp|_{n-1}}{|B_2^{n-1}|}
 =\frac{1}{(n-1)|B_2^{n-1}|}
  \int_{\Sph^{n-1}\cap\xi^\perp}
       \rho_{L_\varepsilon}(v)^{n-1}\,\dd v.
\end{equation}
For $n\geq3$, the integral is the spherical Radon transform.  It is
unchanged when $\rho_{L_\varepsilon}^{n-1}$ is replaced by its even part,
because the equator is invariant under $v\mapsto-v$.  The spherical Radon
transform is continuous on the space of even $C^\infty$ functions
\cite[Theorem~C.2.5]{GardnerBook}.  Together with
\eqref{eq:nonsym-radial-from-support}, this shows, also after
differentiation in $\varepsilon$, that $\mathcal A$ is smooth in
$(\varepsilon,\xi)$.  When $n=2$, the integral is the sum over the two
points of $\Sph^0$, and the same conclusion follows directly.

Rotational invariance shows that $\mathcal A(\varepsilon,\xi)$ depends only
on $\xi_1$; denote the resulting profile by $A(\varepsilon,\xi_1)$.
Because $\xi^\perp=(-\xi)^\perp$ and
$L_{-\varepsilon}=-L_\varepsilon$, we have
\begin{equation}\label{eq:nonsym-A-parity}
 A(\varepsilon,-\xi_1)=A(\varepsilon,\xi_1),
 \qquad
 A(-\varepsilon,\xi_1)=A(\varepsilon,\xi_1).
\end{equation}
Also $A(0,\xi_1)=1$.  If $\xi_1=0$, then $\xi^\perp$ contains the
axis of revolution.  On every slice perpendicular to that axis,
intersection with $\xi^\perp$ and orthogonal projection onto
$\xi^\perp$ produce the same centered ball.  Integrating the slices and
using \eqref{eq:nonsym-projection} gives
\begin{equation}\label{eq:nonsym-A-axis}
 A(\varepsilon,0)=1.
\end{equation}
For $n=2$, the same assertion says that intersection with and projection
onto the axis give the same interval.

Set $F=A-1$.  On $|\xi_1|\leq\frac12$, the profile is smooth in
$(\varepsilon,\xi_1)$; for example, evaluate $\mathcal A$ at the smoothly
varying normal $
 \xi=(\xi_1,\sqrt{1-\xi_1^2},0,\ldots,0)$.
The parity relations \eqref{eq:nonsym-A-parity}, together with
$F(0,\xi_1)=F(\varepsilon,0)=0$, give, by Taylor's formula with integral
remainder,
\[
 F(\varepsilon,\xi_1)
 =\varepsilon^2\xi_1^2\int_0^1\!\int_0^1
 (1-r)(1-s)
 \partial_\varepsilon^2\partial_{\xi_1}^2
 F(r\varepsilon,s\xi_1)\,\dd s\,\dd r.
\]
The mixed derivative is uniformly bounded on a sufficiently small compact
parameter set, so
\begin{equation}\label{eq:nonsym-local-factor}
 |F(\varepsilon,\xi_1)|\leq C_1\varepsilon^2\xi_1^2
 \qquad\text{when }|\xi_1|\leq\frac12.
\end{equation}
On the other hand, the smoothness of $\mathcal A$ on the compact set
$[-\varepsilon_0,\varepsilon_0]\times\Sph^{n-1}$, its evenness in
$\varepsilon$, and $\mathcal A(0,\xi)=1$ give the uniform estimate
\[
 |\mathcal A(\varepsilon,\xi)-1|\leq C_2\varepsilon^2.
\]
If $|\xi_1|\geq\frac12$, the last bound is at most
$4C_2\varepsilon^2\xi_1^2$.  Taking
$M=\max\{C_1,4C_2\}$ proves
\eqref{eq:nonsym-section-factorization}.  
\end{proof}

\begin{proof}[Proof of Theorem~\ref{thm:nonsymmetric}]
Let $M$ be as in Lemma~\ref{lem:nonsym-section-factorization}.  Since
\[
 \lim_{x\downarrow0}
 \frac{1-(1-x)^{(n-1)/2}}{x}=\frac{n-1}{2},
\]
there is $x_0\in(0,1)$ such that
\begin{equation}\label{eq:nonsym-elementary}
 (1-x)^{(n-1)/2}\leq1-\frac{n-1}{4}x\qquad \mbox{ for }
 \qquad 0\leq x\leq x_0.
\end{equation}
Choose $c>0$ so that
\begin{equation}\label{eq:nonsym-c-choice}
 \frac{n-1}{4}c^2>M.
\end{equation}
For $\varepsilon>0$, set
\begin{equation}\label{eq:nonsym-K}
 r_\varepsilon
 =\frac12\left(
  1+\left(\frac{|L_\varepsilon|}{|B_2^n|}\right)^{1/n}
 \right),
 \qquad
 K_\varepsilon=c\varepsilon e_1+r_\varepsilon B_2^n.
\end{equation}
By \eqref{eq:nonsym-volume-gap},
\begin{equation}\label{eq:nonsym-r-range}
 \left(\frac{|L_\varepsilon|}{|B_2^n|}\right)^{1/n}
 <r_\varepsilon<1.
\end{equation}
Also $r_\varepsilon\to1$ as $\varepsilon\to0$.  Fix $\varepsilon>0$ sufficiently small that Lemma~\ref{lem:nonsym-section-factorization}
applies,
\[
 c\varepsilon<r_\varepsilon,
 \qquad c^2\varepsilon^2\leq x_0.
\]
The first inequality ensures that $K_\varepsilon$ contains the origin in
its interior; it also ensures that all the ball sections used below are
nonempty.

Translations do not affect orthogonal projection volumes.  Therefore,
by \eqref{eq:nonsym-projection} and \eqref{eq:nonsym-r-range},
\[
 |K_\varepsilon\mid\xi^\perp|_{n-1}
 =r_\varepsilon^{n-1}|B_2^{n-1}|
 <|B_2^{n-1}|
 =|L_\varepsilon\mid\xi^\perp|_{n-1}
\]
for every $\xi\in\Sph^{n-1}$.  Similarly,
\[
 |K_\varepsilon|
 =r_\varepsilon^n|B_2^n|
 >|L_\varepsilon|.
\]

It remains to compare central sections.  The distance from the center
$c\varepsilon e_1$ of $K_\varepsilon$ to $\xi^\perp$ is
$c\varepsilon|\xi_1|$.  The section is consequently an $(n-1)$-dimensional
ball of radius
$\sqrt{r_\varepsilon^2-c^2\varepsilon^2\xi_1^2}$, and hence
\begin{equation}\label{eq:nonsym-K-section}
 \frac{|K_\varepsilon\cap\xi^\perp|_{n-1}}{|B_2^{n-1}|}
 =\bigl(r_\varepsilon^2-c^2\varepsilon^2\xi_1^2\bigr)^{(n-1)/2}.
\end{equation}
If $\xi_1\ne0$, then \eqref{eq:nonsym-elementary},
$r_\varepsilon<1$, and \eqref{eq:nonsym-c-choice} give
\begin{equation*}
 \frac{|K_\varepsilon\cap\xi^\perp|_{n-1}}{|B_2^{n-1}|}
 \leq
 \bigl(1-c^2\varepsilon^2\xi_1^2\bigr)^{(n-1)/2}\leq1-\frac{n-1}{4}c^2\varepsilon^2\xi_1^2<1-M\varepsilon^2\xi_1^2
 \leq
 \frac{|L_\varepsilon\cap\xi^\perp|_{n-1}}{|B_2^{n-1}|},
\end{equation*}
where the last inequality is precisely
\eqref{eq:nonsym-section-factorization}.  If $\xi_1=0$, the slice--projection
identity used in the proof of that lemma and
\eqref{eq:nonsym-projection} give
\[
 |L_\varepsilon\cap\xi^\perp|_{n-1}=|B_2^{n-1}|,
\]
whereas \eqref{eq:nonsym-K-section} gives
\[
 |K_\varepsilon\cap\xi^\perp|_{n-1}
 =r_\varepsilon^{n-1}|B_2^{n-1}|<|B_2^{n-1}|.
\]
Thus all section and projection inequalities are strict, while
$|K_\varepsilon|>|L_\varepsilon|$.

Both bodies are bodies of revolution and contain the origin in their
interiors.  The support function
\[
 h_{K_\varepsilon}(u)=r_\varepsilon+c\varepsilon u_1
\]
is not even, so $K_\varepsilon$ is not origin-symmetric.  The curvature
function of $L_\varepsilon$ is not even when $\varepsilon\ne0$.  Since surface area
measure is unchanged by translation and is even for every centrally
symmetric body, $L_\varepsilon$ is not centrally symmetric about any point.
This proves Theorem~\ref{thm:nonsymmetric} for every $n\geq2$.
\end{proof}

\section*{Acknowledgments}
The author is grateful to Dmitry Ryabogin for his ICERM lecture
\emph{"On bodies with planar symmetric projections"}, based on ongoing joint
work with Serhii Myroshnychenko and Christos Saroglou.  The idea of using a
spherical harmonic of degree three in the construction of
Section~\ref{sec:nonsymmetric} was inspired by this lecture.

The author used OpenAI's GPT-5.6 Sol to assist with checking
calculations and proof consistency, improving the exposition, and preparing
the manuscript.

\vspace{2mm}
\noindent Artem Zvavitch
\\
Department of Mathematical Sciences, Kent State University, Kent, OH 44242, USA. 
\\ E-mail address: zvavitch@math.kent.edu

\end{document}